\documentclass{article}
\usepackage{geometry}	
\usepackage{amsmath}
\usepackage{hyperref}
\usepackage{amsmath, amsthm, amssymb, mathtools}
\usepackage{amsmath,amssymb,amsfonts,graphicx}
\usepackage{lipsum}
\usepackage{multirow,color}
\usepackage{color}
\newtheorem{lemma}{Lemma}[section]
\newtheorem{corollary}[lemma]{Corollary}
\newtheorem{proposition}[lemma]{Proposition}

\newtheorem{theorem}[lemma]{Theorem}

\theoremstyle{definition}
\newtheorem{remark}[lemma]{Remark}
\newcommand{\e}{\varepsilon}
\newcommand{\f}{\varphi}
\newcommand{\ej}{\varepsilon_j}

\newcommand{\dej}{\delta_j}
\newcommand{\N}{\mathbb{N}}
\newcommand{\R}{\mathbb{R}}
\newcommand{\A}{\mathcal{A}}

\begin{document} 
\title{Multiscale homogenization of non-local energies \\ of convolution-type}
\author{ Giuseppe C. Brusca\footnote{gbrusca@sissa.it}  \\ 
SISSA,  Via Bonomea 265, Trieste, Italy}
\date{}			

\maketitle

\begin{abstract} We analyze a family of non-local integral functionals of convolution-type depending on two small positive parameters $\e,\delta$: the first rules the length-scale of the non-local interactions and produces a `localization' effect as it tends to $0$, the second is the scale of oscillation of a finely inhomogeneous periodic structure in the domain. We prove that a separation of the two scales occurs and that the interplay between the localization and homogenization effects in the asymptotic analysis is determined by the parameter $\lambda$ defined as the limit of the ratio $\e/\delta$. We compute the $\Gamma$-limit of the functionals with respect to the strong $L^p$-topology for each possible value of $\lambda$ and detect  three different regimes, the critical scale being obtained when $\lambda\in(0,+\infty)$.
\end{abstract} 



\section{Introduction}

In their celebrated paper \cite{BBM}, Bourgain, Brezis, and Mironescu proved that a simple approximation of (a multiple of) the $p$-Dirichlet energy is obtained by means of the double integrals 
\begin{equation}\label{functionalsBBM}
    \int_\Omega\int_\Omega \rho_\varepsilon\Bigl(\frac{y-x}{\e}\Bigr)\Bigl|\frac{u(y)-u(x)}{\e}\Bigr|^p\, dx\,dy,
\end{equation}
with $\{\rho_\varepsilon\}_\varepsilon$ a family of radially symmetric non-negative kernels satisfying
\begin{equation*}
\frac{1}{\e^p}\int_{\R^d}\rho_\e(\xi)|\xi|^p\,d\xi=1 \quad \text{ and } \quad   \lim_{\e\to0}\,\frac{1}{\e^p}\int_{\R^d\setminus B_r} \rho_\e(\xi)|\xi|^p\,d\xi = 0
\end{equation*}
for all $\e,r>0$. This result has been extended by Ponce \cite{P1}, also in the sense of $\Gamma$-convergence, to more general homogeneous energies and non-radially symmetric kernels, and, very recently, sharp conditions on the family of kernels $\{\rho_\e\}_\e$ for the validity of such approximation have been detected, see \cite{DDP} and \cite{GS}.

A remarkable class of kernels that fit in this framework is obtained starting from a non-negative $\rho$ having $p$-moment on $\R^d$ equal to $1$ and letting
\begin{equation*}
    \rho_\e(\xi):=\frac{1}{\e^d}\rho\Bigl(\frac{\xi}{\e}\Bigr).
\end{equation*}
Because of this rescaling property, the resulting energies \eqref{functionalsBBM} are called of {\em convolution-type}.

Besides the theory of non-local gradients, that finds several applications in peridynamics (see, e.g., \cite{BM, BMP, MQ}), non-local energies, and in particular those of convolution-type, have been investigated in the last years within many contexts (we refer to \cite{AABPT} for a comprehensive treatment from the variational perspective). For instance, manifold-constrained maps have been considered by Solci \cite{S}, who has obtained an approximation of a vortex energy in the spirit of the Ginzburg-Landau model, and by Giorgio, Happ, and Sch\"onberger \cite{GHS} for the homogenization of micromagnetic energies. Some geometric aspects have been investigated by Berendsen and Pagliari in \cite{BP}, where, together with other results, the asymptotics as $\e\to0$ of the  associated notions of non-local perimeters are established. The analysis of multiscale problems has been addressed by Alicandro, Gelli, and Leone \cite{AGL} in the setting of perforated domains (see also \cite{BP1, BP2}), and by Braides, Scalabrino, and Trifone \cite{BST} in that of disconnected sets.  

The aim of this work is to perform the asymptotic analysis of non-local functionals of convolution-type in the setting of periodic homogenization. Given $\Omega$ a bounded open subset of $\R^d$ with Lipschitz boundary and small positive parameters $\varepsilon, \delta$, we study the functionals 
\begin{equation}\label{functionals0}
\int_\Omega\int_\Omega \frac{1}{\e^d}\rho\Bigl(\frac{y-x}{\e}\Bigr)f\Bigl(\frac{x}{\delta},\frac{y}{\delta},\frac{u(y)-u(x)}{\e}\Bigr)\, dx\,dy, \quad u\in L^p(\Omega; \R^m),
\end{equation}
for $p\in(1,+\infty)$, under general assumptions on $\rho$, the interaction kernel, and the density $f$, that is supposed to be $Q_1$-periodic in the first two variables. 

As already mentioned, the parameter $\e$ is responsible for a {\em localization} of the functionals \eqref{functionals0}, which possess a finite limit as $\e\to0$ only on Sobolev functions. On the other hand, since $f$ is periodic, the above energies encode some average properties of a finely inhomogeneous structure when $\delta$ is small, which, in broad terms, yields a {\em homogenized} limit energy independent of the spatial variables. In light of these observations, it is expected that both the localization and homogenization phenomena may be exhibited by our model when the parameters $\e, \delta$ vanish simultaneously; and therefore, it is rather natural to ask how these effects combine.

For this reason, we assume that $\delta=\delta(\e)$ vanishes as $\e\to0$, and prove that a {\em separation of the scales} $\e$ and $\delta$ occurs, providing a complete description of the effective limit (in the sense of $\Gamma$-convergence with respect to the strong $L^p$-topology) in accordance with the (possibly different) rates of convergence to $0$ of the involved parameters. 

\smallskip

In order to illustrate our result, we consider a simplified, but prototypical, example of non-local oscillating energies given by
\begin{equation}\label{functionalproto}
  F_{\e,\delta}(u):=\int_\Omega\int_\Omega \frac{1}{\e^d}\rho\Bigl(\frac{y-x}{\e}\Bigr)a\Bigl(\frac{x}{\delta}\Bigr)\Bigl|\frac{u(y)-u(x)}{\e}\Bigr|^p\, dx\,dy,
\end{equation}
that is obtained from \eqref{functionals0} upon setting $f(x,y,z)=a(x)|z|^p$, with the function $a$ that is $Q_1$-periodic and such that $0<\alpha\leq a(x)\leq \beta<+\infty$ for a.e. $x\in \R^d$. 

As a starting point for our study, we may consider the case $\delta(\e)=\e$ that has already been treated in \cite{AABPT}. In this instance, it is proved that
\begin{equation*}
  \Gamma(L^p)\text{-}\lim_{\e\to0} F_{\e,\e}(u)  = \int_\Omega f_{\rm hom}^{\rm NL}(\nabla u)\, dx, \quad u\in W^{1,p}(\Omega;\R^m),
\end{equation*}
 the integrand of the homogenized energy $f_{\rm hom}^{\rm NL}$ being characterized through a so-called {\em non-local cell-problem formula} (see \cite[Theorem 6.2]{AABPT}) given by 
\begin{equation*}
      f_{\rm hom}^{\rm NL}(M)= \inf\Bigl\{ \int_{\R^d}\int_{Q_1} \rho(y-x) a(x)|u(y)-u(x)|^p\, dx\, dy : u\in L^p_{\#,M}(Q_1;\R^m) \Bigr\}
\end{equation*}
for all $M\in \R^{m\times d}$, where
\begin{equation*}
    L^p_{\#,M}(Q_1; \R^m):=\{u\in L^{p}_{\rm loc}(\R^d;\R^m) : u-Mx \text{ is } Q_1\text{-periodic}\}.
\end{equation*}
It is immediate to extend this result to the case that  $\delta$ is a multiple of $\e$
\begin{equation*}
    \lambda\delta(\e)=\e, \quad \e>0,
\end{equation*}
for some $\lambda\in(0,+\infty)$, obtaining that 
\begin{equation}\label{critical}
  \Gamma(L^p)\text{-}\lim_{\e\to0} F_{\e,\frac{\e}{\lambda}}(u)  = \int_\Omega f_{\rm hom, \lambda}^{\rm NL}(\nabla u)\, dx,
\end{equation}
where now
\begin{equation*}
      f_{\rm hom, \lambda}^{\rm NL}(M)= \inf\Bigl\{ \int_{\R^d}\int_{Q_1} \frac{1}{\lambda^d}\rho\Bigl(\frac{y-x}{\lambda}\Bigr) a(x)\Bigl|\frac{u(y)-u(x)}{\lambda}\Bigr|^p\, dx\, dy : u\in L^p_{\#,M}(Q_1;\R^m) \Bigr\}.
\end{equation*}

Different limits are obtained if we let the parameters $\e,\delta$ tend to $0$ separately. To see this, we suppose for simplicity of exposition that the kernel $\rho$ is radial and that the coefficient $a$ is continuous. If we let first $\e\to0$ (keeping $\delta$ fixed), applying \cite[Corollary 8]{P1} we obtain that 
\begin{equation*}
    F_\delta(u):=\Gamma(L^p)\text{-}\lim_{\e\to0} F_{\e,\delta}(u)= \kappa \int_\Omega a\Bigl(\frac{x}{\delta}\Bigr)|\nabla u(x)|^p\, dx, \quad u\in W^{1,p}(\Omega; \R^m),
\end{equation*}
 where
\begin{equation}\label{kappa}
\kappa:=\int_{\R^d}\rho(\xi)|\xi_1|^p\,d\xi.
\end{equation}
Then, letting $\delta\to0$ and using a known result in $\Gamma$-convergence for the homogenization of integral functionals (see \cite[Theorem 14.7]{BDF}) we infer
\begin{equation*}
  \Gamma(L^p)\text{-}\lim_{\delta\to0}  F_\delta(u)= \kappa \int_\Omega f_{\rm hom}(\nabla u)\, dx,
\end{equation*}
where the integrand is described by the {\em cell-problem formula}
\begin{equation*}
    f_{\rm hom}(M):=\inf\Bigl\{\int_{Q_1}a(x)|\nabla u(x)|^p\,dx: u \in W^{1,p}_{\#,M}(Q_1; \R^m) \Bigr\}
\end{equation*}
and
\begin{equation*}
    W^{1,p}_{\#,M}(Q_1; \R^m):=\{u\in W^{1,p}_{\rm loc}(\R^d;\R^m) : u-Mx \text{ is } Q_1\text{-periodic}\}.
\end{equation*}

In a similar fashion, if we first let $\delta\to0$, by the periodicity of the coefficient $a$  and the Riemann-Lebesgue Lemma, we have that
\begin{equation*}
    F_\e(u):=\Gamma(L^p)\text{-}\lim_{\delta\to0} F_{\e,\delta}(u)=  \Bigl(\int_{Q_1}a\,dx\Bigr)\int_\Omega\int_\Omega \frac{1}{\e^d}\rho\Bigl(\frac{y-x}{\e}\Bigr)\Bigl|\frac{u(y)-u(x)}{\e}\Bigr|^p\, dx\,dy;
\end{equation*}
and then, applying once again \cite[Corollary 8]{P1}, as $\e\to0$ we obtain
\begin{equation*}
  \Gamma(L^p)\text{-}\lim_{\e\to0}  F_\e(u)= \kappa \Bigl(\int_{Q_1}a\,dx\Bigr)\int_\Omega |\nabla u|^p\, dx, \quad u\in W^{1,p}(\Omega; \R^m),
\end{equation*}
with $\kappa$ as in \eqref{kappa}.

Since functionals \eqref{functionalproto} are equi-coercive in the strong $L^p$-topology on functions with fixed mean value on $Q_1$ (see \cite[Theorems 1.1, 1.2]{P2}), using, at least formally, a diagonal argument (see \cite{DM} for the metrizability of $\Gamma$-convergence) we obtain that there exist two scales $\delta'(\e)$ and $\delta''(\e)$ such that 
\begin{equation*}
    \Gamma(L^p)\text{-}\lim_{\e\to0} F_{\e,\delta'(\e)}(u)=  \kappa \int_\Omega f_{\rm hom}(\nabla u)\, dx
\end{equation*}
and
\begin{equation*}
    \Gamma(L^p)\text{-}\lim_{\e\to0} F_{\e,\delta''(\e)}(u)=  \kappa \Bigl(\int_{Q_1}a\,dx\Bigr)\int_\Omega |\nabla u|^p\, dx
\end{equation*}
whenever $u\in W^{1,p}(\Omega; \R^m)$. These two behaviours differ from that obtained in \eqref{critical}, where $\e/\delta(\e)$ is a fixed positive constant, suggesting that the $\Gamma$-limit of $\{F_{\e,\delta(\e)}\}_\e$ may be governed by the new parameter
\begin{equation*}
\lambda:=\lim_{\e\to0}\frac{\e}{\delta(\e)}\in[0,+\infty],
\end{equation*}
and in particular that the critical scaling may be obtained for $\lambda\in(0,+\infty)$. This is the content of our main result. In view of its statement, we set our standing assumptions. 

\smallskip

We suppose that $\rho:\R^d\to [0,+\infty]$ is a non-negative Borel function (not necessarily radially symmetric) such that there exist $c_0,r_0>0$ with the property that 
\begin{equation}\label{rho : ball}
\tag{$\rho_1$}
    \rho(\xi)\geq c_0 \quad \text{ for almost every } \xi\in B_{r_0},
\end{equation}
and 
\begin{equation}\label{rho : moment}
\tag{$\rho_2$}
    \int_{\R^d}\rho(\xi)|\xi|^p\, d\xi <+\infty,
\end{equation}
then we suppose that $f: \R^d\times \R^d\times \R^m\to[0,+\infty)$ is a Borel function such that 
\begin{align}  \notag
& f(\cdot,y,z) \text{ and } f(x,\cdot,z)  \text{ are } Q_1\text{-periodic for every } z\in \R^m \\ \label{f : period} \tag{P}
& \qquad\text{ and for almost every } y, x\in\R^d, \text{ respectively},
\end{align}
\begin{equation} \label{f : convex}
    \tag{C}
f(x,y,\cdot) \text{ is convex for almost every } x,y\in \R^d,
\end{equation}
and that fulfills the growth conditions 
\begin{equation} \label{f : growth}
    \tag{GC}
\alpha|z|^p\leq f(x,y,z)\leq \beta|z|^p \text{ for almost every } x,y\in \R^d \text{ and for every } z\in \R^m
\end{equation}
for some strictly positive numbers $\alpha,\beta$ and an exponent $p\in(1,+\infty)$. 

Besides these hypotheses, according to the value of $\lambda$ we shall assume that
\begin{equation} \label{f : continuous}
    \tag{H0}
f(x,\cdot,z) \text{ is continuous for almost every } x\in \R^d \text{ and for every } z\in \R^m,
\end{equation}
or 
\begin{equation}\label{H1}
       \tag{H1} \rho(\xi)f(x,y,z)=\rho(-\xi)f(x,y,-z)  \text{ for almost every } \xi,x,y\in\R^d \text{ and for every } z\in \R^m, 
        \end{equation}
    or that 
\begin{equation}\label{H2}
    \tag{H2}
    \int_{\R^d}\rho(\xi)\,d\xi =+\infty.
\end{equation}
Note that if $f(x,y,z)=f(y,x,z)$ and $f(x,y,z)=f(x,y,-z)$ for a.e. $x,y\in\R^d$ and for all $z\in\R^m$, it is not restrictive to suppose that $\rho(\xi)=\rho(-\xi)$ for a.e. $\xi\in\R^d$, see, e.g.,  \cite[Remark 2.1]{AABPT}.

\smallskip

We use the change of variables $\xi:=(y-x)/\e$ and introduce the notation $E_{\e}(\xi):=\{x\in E : x+\e\xi\in E\}$ to rewrite the functional \eqref{functionals0} as
    \begin{equation*}
\int_{\R^d}\rho(\xi)\int_{\Omega_{\e}(\xi)} f\Bigl(\frac{x}{\delta},\frac{x+\e\xi}{\delta},\frac{u(x+\e\xi)-u(x)}{\e}\Bigr)\, dx\,d\xi;
\end{equation*}
then, we consider $\{\e_j\}_j$ and $\{\dej\}_j$ two sequences of positive numbers converging to $0$ as $j\to+\infty$ and define the functionals $F_j: L^p(\Omega; \R^m)\to[0,+\infty]$ as
\begin{equation}\label{functionals}
    F_j(u)  := \int_{\R^d} \rho(\xi)\int_{\Omega_{\ej}(\xi)} f\Bigl(\frac{x}{\dej},\frac{x+\ej\xi}{\dej},\frac{u(x+\ej\xi)-u(x)}{\ej}\Bigr)\, dx\,d\xi, \quad j\in\N.
\end{equation}

\begin{theorem}\label{thm : main}
    Let $\{\ej\}_j$ and $\{\dej\}_j$ be sequences such that $\ej\to 0^+$ and $\dej\to0^+$ as $j\to+\infty$, and assume there exists
    \begin{equation*}
        \lambda:=\lim_{j\to+\infty}\frac{\ej}{\dej}\in[0,+\infty].
    \end{equation*} Let $\rho$ be a non-negative kernel satisfying \eqref{rho : ball} and \eqref{rho : moment}, and let $f$ be a function satisfying \eqref{f : period}, \eqref{f : convex}, and \eqref{f : growth}. In addition, if $\lambda\in[0,+\infty)$ assume that \eqref{f : continuous} is satisfied, and if $\lambda=+\infty$ assume that either \eqref{H1} or \eqref{H2} is satisfied. Then, letting $\{F_j\}_j$ be as in \eqref{functionals}, there exists a quasiconvex function $f_\lambda:\R^{m\times d}\to[0,+\infty)$ such that
    \begin{equation*}
        \Gamma(L^p)\text{-}\lim_{j\to+\infty} F_j(u) = \begin{cases}
            \displaystyle \int_{\Omega } f_\lambda(\nabla u)\,dx & \text{ if } u\in W^{1,p}(\Omega;\R^m), \\
            +\infty & \text{ if } u\in L^p(\Omega;\R^m)\setminus W^{1,p}(\Omega;\R^m).
        \end{cases}
    \end{equation*}
    In particular, the following hold for every $M\in\R^{m\times d}$:
    \begin{itemize}
        \item[(i)] if $\lambda=0$ and \eqref{f : continuous} holds, then
        \begin{equation*}
            f_0(M)= \inf\Bigl\{ \int_{\R^d} \int_{Q_1}\rho(\xi) f(x,x,(\nabla u(x) ) \xi)\, dx\,d\xi : u\in W^{1,p}_{\#,M}(Q_1; \R^m) \Bigr\};
        \end{equation*}
        \item[(ii)] if $\lambda\in(0,+\infty)$ and \eqref{f : continuous} holds, then 
         \begin{equation*}
            f_\lambda(M)= \inf\Bigl\{ \int_{\R^d}\int_{Q_1} \frac{1}{\lambda^d}\rho\Bigl(\frac{y-x}{\lambda}\Bigr) f\Bigl(x,y,\frac{u(y)-u(x)}{\lambda}\Bigr)\, dx\, dy : u\in L^p_{\#,M}(Q_1;\R^m) \Bigr\};
        \end{equation*}
        \item[(iii)] if $\lambda=+\infty$ and either \eqref{H1} or \eqref{H2} holds, then
        \begin{equation*}
            f_{+\infty}(M)=\int_{\R^d} \int_{Q_1}\int_{Q_1}\rho(\xi) f(x,y,M\xi)\, dx\, dy\,d\xi.
        \end{equation*}
    \end{itemize}
\end{theorem}

We remark that the hypotheses \eqref{rho : ball}, \eqref{rho : moment} and \eqref{f : growth} are needed in order to ensure the equi-coerciveness of the functionals $\{F_j\}_j$ so that, as $\e\to0$, these have a finite limit only on $W^{1,p}(\Omega;\R^m)$, see \cite{P2} or \cite{AABPT}, and that the latter can be weakened if we further assume the integrability of the kernel (Remark \ref{rmk}). Assumption \eqref{f : convex} implies the weak lower semicontinuity in $L^p$ of the involved functionals and cannot be removed unconcernedly, since, contrary to the local case, the relaxation procedure for double integrals may even result in the loss of an integral representation as observed in \cite{MT, KZ, B}. We postpone further comments on the conditions \eqref{f : continuous}, \eqref{H1} and \eqref{H2} to the end of the introduction.

\smallskip

The first steps towards the proof of our result are the reduction to the case $\rho$ is supported on a ball followed by the application of a compactness and integral representation result. These results have been proved in \cite{AABPT} in greater generality; we present some simpler versions that are suitable for our purposes, see Lemma \ref{lemma : truncated functionals} and Theorem \ref{thm : intrep}, respectively. The latter states that the sequence of functionals $\{F_j\}_j$ in \eqref{functionals} admits (up to subsequences) a $\Gamma$-limit which has to be finite on $W^{1,p}(\Omega;\R^m)$ only and that is represented in integral form as
\begin{equation*}
  \Gamma(L^p)\text{-}\lim_{j\to+\infty} F_{j}(u)  = \int_\Omega \f(x,\nabla u)\, dx.
\end{equation*}
 As customary in the homogenization of integral functionals, we prove that the integrand $\f$ is independent of the first variable (Lemma \ref{lemma : homogeneity} below). Subsequently, we exploit the quasiconvexity of $\f$ in the gradient variable and the convergence of boundary-value problems for non-local functionals to infer that
\begin{equation*}
      \f(M) = \lim_{j\to+\infty} \inf\{ F_j(u, Q) : u=Mx \text{ close to } \partial Q\},
    \end{equation*}
where we have `localized' the functionals $\{F_j\}_j$ in a cube $Q$ contained in $\Omega$ setting
\begin{equation*}
    F_j(u,Q)  := \int_{\R^d} \rho(\xi)\int_{Q_{\ej}(\xi)} f\Bigl(\frac{x}{\dej},\frac{x+\ej\xi}{\dej},\frac{u(x+\ej\xi)-u(x)}{\ej}\Bigr)\, dx\,d\xi, \quad j\in\N,
\end{equation*}
and the proper meaning of `close to $\partial Q$' shall be clarified in Section $2$. Note that, here, the term `localization' is intended in the sense of the so-called `localization method' for $\Gamma$-convergence in which integral functionals are regarded as set functions as well (see \cite[Chapter 9]{BDF} or \cite[Chapter 5]{AABPT} for the non-local counterpart). Upon manipulating the localized energies using the assumptions \eqref{f : period} and \eqref{f : convex}, together with some rescaling arguments, we obtain that $\f(M)$ is determined in terms of the `non-local cell-problem formulas'
\begin{equation}\label{min_j}
    \inf\Bigl\{ \int_{\R^d} \rho(\xi)\int_{Q_1} f\Bigl(x,x+\frac{\ej}{\dej}\xi,\frac{u(x+\frac{\ej}{\dej}\xi)-u(x)}{\frac{\ej}{\dej}}\Bigr)\, dx\,d\xi : u\in L^p_{\#,M}(Q_1;\R^m) \Bigr\};
\end{equation}
and then, we discuss the asymptotic behaviour of the infima \eqref{min_j} according to the value of the relevant parameter $\lambda$. Our analysis is based on the study of the compactness properties of any sequence $\{u_j\}_j$ having equi-bounded energies; i.e., such that
\begin{equation}\label{intro uniform}
    \sup_j \int_{\R^d} \rho(\xi)\int_{Q_1} f\Bigl(x,x+\frac{\ej}{\dej}\xi,\frac{u_j(x+\frac{\ej}{\dej}\xi)-u_j(x)}{\frac{\ej}{\dej}}\Bigr)\, dx\,d\xi <+\infty,
\end{equation}
and with fixed mean value in $Q_1$ for all $j\in\N$. In particular, if 
    \begin{equation*}
        \lambda=\lim_{j\to+\infty}\frac{\ej}{\dej}= 0, \qquad \qquad (\textit{subcritical case})
    \end{equation*}
    condition \eqref{intro uniform}, together with the assumptions on $\rho$ and $f$, implies the strong convergence (up to subsequences) to a function $u\in W^{1,p}(\Omega;\R^m)$.
    
    If 
    \begin{equation*}
        \lambda=\lim_{j\to+\infty}\frac{\ej}{\dej}\in (0, +\infty), \qquad \qquad (\textit{critical case})
    \end{equation*}
    we infer a uniform bound on $\|u_j\|_{L^p(Q_1;\R^m)}$ that only yields weak compactness in $L^p$ but which is sufficient to carry out our proof.
    
     Finally, when
     \begin{equation*}
        \lambda=\lim_{j\to+\infty}\frac{\ej}{\dej}=+\infty, \qquad \qquad (\textit{supercritical case})
    \end{equation*}
    we are not able to obtain any compactness for $\{u_j\}_j$; we rely on the $L^p$-bounds for the difference quotients
    \begin{equation}\label{diffquot}
        \frac{u_j(x+\frac{\ej}{\dej}\xi)-u_j(x)}{\frac{\ej}{\dej}}, \quad j\in\N,
    \end{equation}
    and on a more delicate argument in the vein of the two-scale convergence \cite{A,N}.

\smallskip

It is interesting to note that, starting from our result, it is possible to infer that the function
\begin{equation*}
   \lambda \mapsto f_\lambda(M)
\end{equation*}
is continuous in $[0,+\infty]$ for all $M\in \R^{m\times d}$, upon assuming \eqref{f : continuous} together with either \eqref{H1} or \eqref{H2}. The continuity at $\lambda=0$, which is formally obtained adapting the equi-coerciveness and $\Gamma$-convergence results in \cite{P1, P2}, can be regarded as a localization of the non-local cell-problem formulas defining $f_\lambda$ to the local cell-problem formula for $f_0$. We rigorously infer such continuity using \eqref{f : continuous}, but this assumption seems to play a merely technical role in the whole work as it is needed to apply a strong-weak lower semicontinuity result in $L^p$-spaces (Theorem \ref{thm : LSC}). Clearly, requiring the continuity in the $y$-variable or in the $x$-variable is equivalent, and is not even needed when only one space-variable is involved; in other words, \eqref{f : continuous} can be removed if $f(x,y,z)=f(x,z)$. 

The case $\lambda=+\infty$ is more complex and a different discussion is needed about the hypotheses \eqref{H1} and \eqref{H2}. According to the heuristics illustrated before for the analysis of \eqref{functionalproto}, it would be natural to conjecture that
\begin{equation*}
f_{+\infty}(M)= \int_{\R^d} \int_{Q_1}\int_{Q_1}\rho(\xi) f(x,y,M\xi)\, dx\, dy\,d\xi
\end{equation*}
for all $M\in \R^{m\times d}$, regardless of the validity of \eqref{H1} or \eqref{H2}. In fact, if we do not enhance our assumptions, we are only able to prove that 
\begin{equation}\label{introlower}
           f_{+\infty}(M)\geq\inf\Bigl\{ \int_{\R^d} \int_{Q_1}\int_{Q_1} \rho(\xi)f(x,x+y,M\xi+V(x,y))\,dx\, dy\,d\xi : V\in\mathcal{C}\Bigr\},
        \end{equation}
where the class $\mathcal{C}$ contains functions periodic in both variables and whose mean value equals $0$. This class is introduced because, roughly speaking, functions of the form $M\xi+V(x,y)$ with $V\in\mathcal{C}$ arise as a two-scale limit of a sequence of difference quotients similar to \eqref{diffquot}. If we assume that either \eqref{H1} or \eqref{H2} holds true, then the minimum problem in \eqref{introlower} is easily seen to be solved only by $V=0$, which allows us to prove $(iii)$ in Theorem \ref{thm : main}. On the contrary, when neither of the hypotheses \eqref{H1} and \eqref{H2} is in force, one can exhibit simple examples in which $V=0$ is not a minimizer for \eqref{introlower}.

This argument suggests that it may be needed a more careful analysis of the difference quotients in \eqref{diffquot}  when $\{u_j\}_j$ is an optimal sequence for the corresponding minimum problems \eqref{min_j}; but it turns out that, for a simple choice of $f$, these difference quotients do not weakly converge to $M\xi$ if both \eqref{H1} and \eqref{H2} are neglected, see Proposition \ref{prop : superobstruction}. We are then led to regard hypotheses \eqref{H1} and \eqref{H2} as necessary to us, when the homogenization occurs at a scale that is infinitesimal compared to that of the localization, in order to ensure a `compatibility' between the two limit processes.

\smallskip

Such a compatibility seems to be lacking, or to be more difficult to exploit, in the fractional setting previously analyzed in collaboration with Braides and Donati \cite{BBD}. For the oscillating energies of fractional-type 
\begin{equation}\label{fractional}
     (1-s)\int_\Omega\int_\Omega a\Bigl(\frac{x}{\delta}\Bigr)\frac{|u(y)-u(x)|^2}{|y-x|^{d+2s}}\, dx\,dy, \quad u\in L^2(\Omega), 
\end{equation}
the localization effect produced by the vanishing of the parameter $\e$ in \eqref{functionals0} is due to the convergence $s\to 1^-$. In that work, we were only able to treat a subcritical case
\begin{equation*}
    \sqrt{1-s} \ll \delta.
\end{equation*}
Although a similar heuristic reasoning to that presented above could be performed in this setting as well, we do not know whether the scale 
\begin{equation*}
    \sqrt{1-s} \sim \delta
\end{equation*}
is critical or not; that is, whether three regimes only are allowed in the limit for the functionals \eqref{fractional}. Such functionals do not fit in the framework of convolution-type energies because the dependence of \eqref{fractional} on $1-s$ is different from the dependence of functionals \eqref{functionals0} on the parameter $\e$. This complicates the analysis already in the case $\sqrt{1-s} \ll \delta$, which, in fact, is studied by means of a discretization argument.

\section{Notation and preliminary results}

In this section we fix the notation and state some preliminary results.

We let $d,m$ be two positive integers and consider the Euclidean space $\R^k, k\in\{d,m, m\times d\}$. We identify $\R^d$ with the ambient space containing the domain $\Omega$, a bounded open set with Lipschitz boundary, and $\R^m$ with the target space of the vector-valued functions we consider, usually in $L^p_{\rm loc}(\Omega;\R^m)$ or $W_{\rm loc}^{1,p}(\Omega;\R^m)$. For $k\in\{d,m, m\times d\}$, we let $|\xi|$ denote the Euclidean norm of a vector $\xi\in \R^k$. For $k=d$, we let $B_r(x)$ denote the $d$-dimensional Euclidean ball of centre $x$ and radius $r$ and we let $Q_r$ denote the $d$-dimensional cube of side-length $r$ given by $Q_r:=(0,r)^d$.

Given $t\in \R$, we let $\lfloor t \rfloor$ and $\lceil t \rceil$ denote the lower and the upper integer part of $t$, respectively. Analogously, given $\xi\in \R^d,$ we set $\lfloor \xi \rfloor:=(\lfloor \xi_1 \rfloor,...,\lfloor \xi_d \rfloor)\in \mathbb{Z}^d$ and $\lceil\xi\rceil:=(\lceil \xi_1 \rceil,...,\lceil \xi_d \rceil)\in \mathbb{Z}^d$. We say that a set $E\subseteq \R^k, k\in\{d,m\}$, is measurable provided that it is measurable with respect to the $k$-dimensional Lebesgue measure and, in this case, we let $|E|$ denote its measure.

We say that a function $u: \R^d\to \R^m$ is $Q_1$-periodic if, letting $\{e_1,...,e_d\}$ denote the canonical basis of $\R^d$, it holds
\begin{equation*}
    u(x+e_i)=u(x)
\end{equation*}
for almost every $x\in \R^d$ and for all $i\in\{1,...,d\}$. We recall that, for any $M\in \R^{m\times d}$, the following class of functions have been defined in the introductory section:
\begin{equation*}
    L^p_{\#,M}(Q_1; \R^m):=\{u\in L^{p}_{\rm loc}(\R^d;\R^m) : u-Mx \text{ is } Q_1\text{-periodic}\}
\end{equation*}
and
\begin{equation*}
    W^{1,p}_{\#,M}(Q_1; \R^m):=\{u\in W^{1,p}_{\rm loc}(\R^d;\R^m) : u-Mx \text{ is } Q_1\text{-periodic}\}.
\end{equation*}
In order to treat boundary-value problems in a non-local setting, we introduce the class of functions
\begin{equation*}
    \mathcal{D}_{r,M}(A; \R^m):=\{u\in L^{p}(A;\R^m) : u(x)=Mx \,\text{ for a.e. } x\in A, \text{dist}(x, \R^d\setminus A)<r\},
\end{equation*}
with $A$ an open subset of $\R^d, M\in \R^{m\times d},$ and $r>0$.

Given $A\subseteq\R^d$ open, we say that a sequence of functionals $\mathcal{F}_j: L^p(A;\R^m)\to[0,+\infty], j\in\N$, $\Gamma$-converges with respect to the strong $L^p(A;\R^m)$-topology to a functional $\mathcal{F}: L^p(A;\R^m)\to[0,+\infty]$ as $j\to+\infty$ if the following hold:
\begin{itemize}
    \item[(i)] for every $u\in L^p(A;\R^m)$ and $\{u_j\}_j\subset L^p(A;\R^m)$ such that $u_j\to u$ in $L^p(A;\R^m)$ as $j\to+\infty$ it holds
    \begin{equation*}
        \liminf_{j\to+\infty} \mathcal{F}_j(u_j) \geq \mathcal{F}(u);
    \end{equation*}
    \item[(ii)] for every $u\in L^p(A;\R^m)$ there exists a sequence $\{v_j\}_j\subset L^p(A;\R^m)$ such that $v_j\to u$ in $L^p(A;\R^m)$ as $j\to+\infty$ and it holds
    \begin{equation*}
        \limsup_{j\to+\infty} \mathcal{F}_j(v_j) \leq \mathcal{F}(u).
        \end{equation*}
\end{itemize}
In this case, we write
\begin{equation*}
    \Gamma(L^p)\text{-}\lim_{j\to+\infty} \mathcal{F}_j(u)=\mathcal{F}(u)
    \end{equation*}
for all $u\in L^p(A;\R^m)$.

Throughout the whole work, we shall always assume that the kernel $\rho:\R^d\to[0,+\infty]$ is a Borel function that satisfies \eqref{rho : ball} and \eqref{rho : moment}, and that the density $f:\R^d\times\R^d\times\R^m\to[0,+\infty)$ is a Borel function that satisfies \eqref{f : period}, \eqref{f : convex}, and \eqref{f : growth}. In some specific cases, such assumptions will be enhanced through \eqref{f : continuous}, \eqref{H1}, and \eqref{H2}.

Given $T>0$ and letting $\A_{\rm reg}(\Omega)$ denote the family of open subsets of $\Omega$ with Lipschitz boundary, we introduce for all $j\in \N$ the truncated and `localized' versions of the functionals \eqref{functionals} that are defined as $F_j^T: L^p(\Omega; \R^m)\times \A_{\rm reg}(\Omega)\to[0,+\infty]$,
\begin{equation}\label{truncatedfunctionals}
    F^T_j(u,A)  := \int_{B_T} \rho(\xi)\int_{A_{\ej}(\xi)} f\Bigl(\frac{x}{\dej},\frac{x+\ej\xi}{\dej},\frac{u(x+\ej\xi)-u(x)}{\ej}\Bigr)\, dx\,d\xi,
\end{equation}
 where for all $\e>0$ and $\xi\in \R^d$ we adopt the notation
\begin{equation*}
    A_\e(\xi):=\{x\in A : x+\e\xi\in A\}.  
    \end{equation*}
If $A=\Omega$ we simply write $F_j^T(u)$ in place of $F_j^T(u,\Omega)$. The notation for the localized functionals is immediately adapted to the non-truncated energies $\{F_j\}_j$.

Now we recall some useful facts about $\Gamma$-convergence in the non-local context. In each of the following statements, the $\Gamma$-limit is intended with respect to the strong convergence in $L^p(\Omega;\R^m)$. 

The first fact concerns the possibility of simplifying the asymptotic analysis by considering kernels whose support is a ball, see \cite[Lemma 5.1]{AABPT}.

\begin{lemma}\label{lemma : truncated functionals} Assume there exist a sequence $\{T_h\}_h$ monotonically increasing to $+\infty$ and functionals $\{F^{T_h}(\cdot,\cdot)\}_h$ such that for all $h\in\N$ we have
\begin{equation*}
    \Gamma(L^p)\text{-}\lim_{j\to+\infty} F^{T_h}_j(u,A) = F^{T_h}(u,A) 
\end{equation*}
for every $u\in L^p(\Omega; \R^m)$ and $A\in \A_{\rm reg}(\Omega)$. Then
\begin{equation*}
   \Gamma(L^p)\text{-}\lim_{j\to+\infty} F_j(u,A) =\lim_{h\to+\infty} F^{T_h}(u,A) 
\end{equation*}
for every $u\in L^p(\Omega; \R^m)$ and $A\in \A_{\rm reg}(\Omega)$.
\end{lemma}

Following \cite[Proposition 6.1]{AABPT}, we prove that if an integral representation is in place for the $\Gamma$-limit of $\{F^T_j\}_j$ with $T>0$, then, as a consequence of the assumption \eqref{f : period}, it holds that the integrand only depends on the gradient variable.

\begin{lemma}\label{lemma : homogeneity} Let $T>0$ and assume there exists a Carathéodory function $\varphi:\Omega\times \R^{m\times d}\to[0,+\infty)$ such that 
\begin{equation*}
     \Gamma(L^p)\text{-}\lim_{j\to+\infty} F^T_j(u,A) = \int_A \f(x, \nabla u)\, dx
\end{equation*}
for every $u\in W^{1,p}(\Omega; \R^m)$ and $A\in\A_{\rm reg}(\Omega)$. Then $\f$ is independent of the first variable.
\end{lemma}

\begin{proof}
    We set 
    \begin{equation*}
        F(u,A):= \int_A \f(x, \nabla u)\, dx
    \end{equation*}
    for $u\in W^{1,p}(\Omega;\R^m)$ and $A\in\A_{\rm reg}(\Omega)$ and we note that the conclusion follows if we prove that 
    \begin{equation*}
        F(Mx, B_r(y))= F(Mx, B_r(y'))
    \end{equation*}
    for every $M\in \R^{m\times d}$ and for every $y,y'\in \R^d, r>0$ such that $B_r(y)$ and $B_r(y')$ are contained in $\Omega$. To prove this claim, it suffices to show that
\begin{equation*}
     F(Mx, B_{r'}(y)) \leq F(Mx, B_r(y'))
\end{equation*}
for every $M,y,y',r$ as above and $r'\in(0,r)$.

Let $\{u_j\}_j$ be such that $u_j\to Mx$ in $L^p(B_r(y'); \R^m)$ as $j\to+\infty$ and
\begin{equation}\label{homogeneity - recovery}
    \lim_{j\to+\infty} F^T_j(u_j, B_r(y')) = F(Mx, B_r(y')).  
\end{equation}
Note that, upon assuming that $j$ is large enough, we have
\begin{equation*}
    B_{r'}(y) -\dej\Bigl\lfloor \frac{y-y'}{\dej}\Bigr\rfloor \subseteq B_r(y')
\end{equation*}
and then, the functions
\begin{equation*}
    v_j(x):= u_j \Bigl(x-\dej\Bigl\lfloor \frac{y-y'}{\dej}\Bigr\rfloor\Bigr)+\dej M\Bigl\lfloor \frac{y-y'}{\dej}\Bigr\rfloor , \quad  j\in\N,
\end{equation*}
are well defined on $B_{r'}(y)$ and such that $v_j\to Mx$ in $L^p(B_{r'}(y); \R^m)$ as $j\to+\infty$. Moreover, by a change of variables and assumption \eqref{f : period}, it holds
\begin{align*}
     F^T_j(v_j, B_{r'}(y)) & = \int_{B_T} \rho(\xi)\int_{(B_{r'}(y))_{\ej}(\xi)} f\Bigl(\frac{x}{\dej},\frac{x+\ej\xi}{\dej},\frac{v_j(x+\ej\xi)-v_j(x)}{\ej}\Bigr)\, dx\,d\xi  \\
     & = \int_{B_T} \rho(\xi)\int_{\bigl(B_{r'}(y)-\dej\bigl\lfloor \frac{y-y'}{\dej}\bigr\rfloor\bigr)_{\ej}(\xi)} f\Bigl(\frac{x}{\dej},\frac{x+\ej\xi}{\dej},\frac{u_j(x+\ej\xi)-u_j(x)}{\ej}\Bigr)\, dx\,d\xi \\
     & \leq F^T_j(u_j, B_{r}(y')).
\end{align*}
Therefore, recalling \eqref{homogeneity - recovery}, we infer
\begin{equation*}
   F(Mx, B_r(y')) = \lim_{j\to+\infty} F^T_j(u_j, B_r(y')) \geq \liminf_{j\to+\infty}  F^T_j(v_j, B_{r'}(y)) \geq F(Mx, B_{r'}(y)),
\end{equation*}
which concludes the proof.
\end{proof}

The following compactness and integral representation result is a consequence of  \cite[Theorem 5.1]{AABPT} and Lemma \ref{lemma : homogeneity}, and constitutes the starting point of our analysis.

\begin{theorem}\label{thm : intrep}
Let $T>0$. There exist a subsequence $\{j_k\}_k\subset \N$ and $\f:\R^{m\times d}\to[0,+\infty)$ a quasiconvex function such that
    \begin{equation*}
    \Gamma(L^p)\text{-}\lim_{k\to+\infty} F^T_{j_k}(u,A) = \begin{cases}
            \displaystyle \int_{A} \f(\nabla u)\,dx & \text{ if } u\in W^{1,p}(A;\R^m), \\
            +\infty & \text{ if } u\in L^p(\Omega;\R^m)\setminus W^{1,p}(A;\R^m)
        \end{cases}
    \end{equation*}
    for every $A\in \A_{\rm reg}(\Omega)$.
\end{theorem}

We recall \cite[ Proposition 5.4]{AABPT}, which establishes the stability of $\Gamma$-convergence for fixed boundary conditions. 

\begin{proposition}\label{prop : boundary} Let $T,s>0$, $M\in\R^{m\times d}$, and assume there exists a functional $F: L^p(\Omega;\R^m)\times \A_{\rm reg}(\Omega)\to[0,+\infty]$ such that 
\begin{equation*}
    \Gamma(L^p)\text{-}\lim_{j\to+\infty} F^T_j(u,A) = F(u,A) 
\end{equation*}
    for every $u\in L^p(\Omega; \R^m)$ and $A\in \A_{\rm reg}(\Omega)$. Then
    \begin{equation*}
        \lim_{j\to+\infty}\inf\{F^T_j(u,A) : u\in \mathcal{D}_{s\e_j,M}(A; \R^m)\} = \inf \{F(u,A) : u-Mx\in W^{1,p}_0(A; \R^m)\}
    \end{equation*}
    for every $A\in \A_{\rm reg}(\Omega)$.
\end{proposition}

To conclude, we state a simplified version of a useful result about the lower semicontinuity of integral functionals with respect to the strong-weak convergence in $L^p\times L^p$, see \cite[Theorem 7.5]{FL}.

\begin{theorem}\label{thm : LSC}
    Let $p\in(1,+\infty)$, $\ell\in \N^+,$ $E\subseteq \R^\ell$ be measurable,
and $\Psi: E\times \R^d\times \R^m\to[0,+\infty)$ be a Borel function such that
\begin{equation*}
    \Psi(s,\cdot,\cdot) \text{ is continuous for almost every } s\in E
\end{equation*}
and 
\begin{equation*}
    \Psi(s,t,\cdot) \text{ is convex for almost every } s\in E \text{ and for every } t\in\R^d.
\end{equation*}
Then the functional
\begin{equation*}
    (v,u)\in L^p(E; \R^d)\times L^p(E; \R^m)\mapsto \int_{E} \Psi(s,v(s), u(s))\,ds
\end{equation*}
is sequentially weakly lower semicontinuous with respect to the strong-weak convergence in $L^p(E; \R^d)\times L^p(E; \R^m)$.
\end{theorem}

\section{Asymptotic analysis of the truncated functionals}

This section is devoted to the asymptotic analysis of functionals \eqref{truncatedfunctionals} for which the interaction kernel $\rho$ is supported on a ball. This analysis, in combination with Lemma \ref{lemma : truncated functionals}, shall allow us to prove the main result in the final section.

We begin with two preliminary lemmas.

\begin{lemma}\label{lemma : diadic} Let $M\in \R^{m\times d}$, $\{u_j\}_j\subset L^p_{\#,M}(Q_1; \R^m)$, $\{\lambda_j\}_j\subset\R$, and $v_j(x):=u_j(x)-Mx, j\in\N$. For every $E\subset \R^d$ measurable bounded set and $r>0$, there exists a positive constant $C$ depending on $d,p, r$, and $E$ such that
    \begin{equation*}
         \int_{E} \int_{Q_1}\Bigl|\frac{v_j(x+\lambda_j\xi)-v_j(x)}{\lambda_j}\Bigr|^p \, dx\,d\xi \leq C \int_{B_{r}} \int_{Q_1}\Bigl|\frac{v_j(x+\lambda_j\xi)-v_j(x)}{\lambda_j}\Bigr|^p \, dx\,d\xi
    \end{equation*}
    for every $j\in\N$. In particular, if we assume that  
    \begin{equation*}
        \sup_{j}  \int_{B_{r}} \int_{Q_1}\Bigl|\frac{u_j(x+\lambda_j\xi)-u_j(x)}{\lambda_j}\Bigr|^p \, dx\,d\xi <+\infty
    \end{equation*}
    for some $r>0$, then
\begin{equation*}
        \sup_{j}  \int_{E} \int_{Q_1}\Bigl|\frac{u_j(x+\lambda_j\xi)-u_j(x)}{\lambda_j}\Bigr|^p \, dx\,d\xi <+\infty.
    \end{equation*}
    \end{lemma}
\begin{proof}
It suffices to prove the statement in the case $E=B_R$. If $R\leq r$ this is obviously true; therefore, we prove the statement in the case $R=2r$, so that the general case follows by a simple dyadic argument. 

Using the changes of variables $\xi':=\xi/2$ and $x':=x+\lambda_j\xi'$, for all $j\in\N$ we have
    \begin{align*}
\int_{B_{2r}}\int_{Q_1}\Bigl|\frac{v_j(x+\lambda_j\xi)-v_j(x)}{\lambda_j}\Bigr|^p\, dx\, d\xi 
        & = 2^d \int_{B_{r}}\int_{Q_1}\Bigl|\frac{v_j(x+2\lambda_j\xi')-v_j(x)}{\lambda_j}\Bigr|^p\, dx\, d\xi' \\
        & \leq 2^{d+p-1} \Bigl\{ \int_{B_{r}}\int_{Q_1}\Bigl|\frac{v_j(x+2\lambda_j\xi')-v_j(x+\lambda_j\xi')}{\lambda_j}\Bigr|^p\, dx\, d\xi' \\
        & \quad + \int_{B_{r}}\int_{Q_1}\Bigl|\frac{v_j(x+\lambda_j\xi')-v_j(x)}{\lambda_j}\Bigr|^p\, dx\, d\xi' \Bigr\} \\
        & = 2^{d+p-1}\Bigl\{\int_{B_{r}}\int_{Q_1+\lambda_j\xi'}\Bigl|\frac{v_j(x'+\lambda_j\xi')-v_j(x')}{\lambda_j}\Bigr|^p\, dx'\, d\xi' \\
        & \quad + \int_{B_{r}}\int_{Q_1}\Bigl|\frac{v_j(x+\lambda_j\xi')-v_j(x)}{\lambda_j}\Bigr|^p\, dx\, d\xi' \Bigr\}.
    \end{align*}
Since $v_j$ is $Q_1$-periodic, the function
\begin{equation*}
    x'\mapsto \Bigl|\frac{v_j(x'+\lambda_j\xi')-v_j(x')}{\lambda_j}\Bigr|^p
\end{equation*}
is $Q_1$-periodic as well for every $\xi'\in \R^d$, hence,
\begin{equation*}
\int_{Q_1+\lambda_j\xi'}\Bigl|\frac{v_j(x'+\lambda_j\xi')-v_j(x')}{\lambda_j}\Bigr|^p\, dx' = \int_{Q_1}\Bigl|\frac{v_j(x'+\lambda_j\xi')-v_j(x')}{\lambda_j}\Bigr|^p\, dx'
\end{equation*}
    and then 
    \begin{equation*} \int_{B_{2r}}\int_{Q_1}\Bigl|\frac{v_j(x+\lambda_j\xi)-v_j(x)}{\lambda_j}\Bigr|^p\, dx\, d\xi \leq 2^{d+p}\int_{B_{r}}\int_{Q_1}\Bigl|\frac{v_j(x+\lambda_j\xi)-v_j(x)}{\lambda_j}\Bigr|^p\, dx\, d\xi,
    \end{equation*}
    which proves the first part of the statement.

    The second part of the statement is an immediate consequence. 
\end{proof}

The following lemma provides some precompactness results for sequences having equi-bounded energy.

\begin{lemma} \label{lemma : compactness}
    Let $M\in \R^{m\times d}$, $\{u_j\}_j\subset L^p_{\#,M}(Q_1;\R^m)$, $\{\lambda_j\}_j\subset \R^+$ such that $\lambda_j\to \lambda\in[0,+\infty)$ as $j\to+\infty$, and assume there exists $A$ a bounded open subset of $\R^d$ with Lipschitz boundary such that $Q_1\subseteq A$ and
    \begin{equation*}
        \sup_{j} \Bigl\{ \int_{Q_1} u_j\, dx + \int_{B_r}\int_{A}\Bigl|\frac{u_j(x+\lambda_j\xi)-u_j(x)}{\lambda_j}\Bigr|^p\, dx\, d\xi \Bigr\}<+\infty
    \end{equation*}
    for some $r>0$.
    The following hold:
    \begin{itemize}
        \item[(i)] if $\lambda=0$, there exist a subsequence $\{u_{j_k}\}_k$ and a function $u\in W^{1,p}(A;\R^m)$ such that $u_{j_k}\to u$ strongly in $L^p(A;\R^m)$ as $k\to+\infty$;
        \item[(ii)] if $\lambda\in(0,+\infty)$, there exist a subsequence $\{u_{j_k}\}_k$ and a function $u\in L^{p}(A;\R^m)$ such that $u_{j_k}\rightharpoonup u$ weakly in $L^p(A;\R^m)$ as $k\to+\infty$;
    \end{itemize}
    in both cases, it holds that $\sup_j \|u_j\|_{L^p(Q_1;\R^m)}$ is finite.
\end{lemma}

\begin{proof}
    The statement $(i)$ readily follows by \cite[Corollary $4.2$]{AABPT}.

    As for the proof of $(ii)$, we observe that, by the translation invariance of the functional, it is not restrictive to suppose that
    \begin{equation}\label{mean value}
        \int_{Q_1} u_j\,dx=0
    \end{equation}
    and that $\lambda/2\leq\lambda_j\leq 2\lambda$ for every $j\in\N$; moreover, since $\{u_j\}_j\subset L^p_{\#,M}(Q_1;\R^m)$, we may further assume that $A=Q_1$.

Using Lemma \ref{lemma : diadic} with $E=Q_{2/\lambda}$, we get 
\begin{equation*}
    S:=\sup_j\int_{Q_{2/\lambda}} \int_{Q_1} \Bigl|\frac{u_j(x+\lambda_j\xi)-u_j(x)}{\lambda_j}\Bigr|^p\, dx\,d\xi <+\infty;
\end{equation*}
hence, applying the change of variables $\xi':=\lambda_j\xi$, we obtain
\begin{align} \nonumber
    (2\lambda)^pS & \geq \sup_j \int_{Q_{2/\lambda}} \int_{Q_1} |u_j(x+\lambda_j\xi)-u_j(x)|^p\, dx\,d\xi \\ \nonumber
    & = \sup_j \lambda_j^{-d}\int_{Q_{\frac{2\lambda_j}{\lambda}}} \int_{Q_1} |u_j(x+\xi')-u_j(x)|^p\, dx\,d\xi' \\ \label{criticalbound}
    & \geq (2\lambda)^{-d} \sup_j \int_{Q_1} \int_{Q_1} |u_j(x+\xi')-u_j(x)|^p\, dx\,d\xi'.
\end{align}
Resorting to \eqref{mean value} and to the fact that $u_j(x)-Mx$ is $Q_1$-periodic, we apply Jensen's inequality to infer
\begin{align*}
    \int_{Q_1} |u_j(x)|^p\,dx & = \int_{Q_1} \Bigl|u_j(x)-\int_{Q_1}u_j(\xi)\, d\xi\Bigr|^p\,dx \\
    & = \int_{Q_1} \Bigl|\int_{Q_1}u_j(x)-u_j(x+\xi)+u_j(x+\xi)-u_j(\xi)\, d\xi\Bigr|^p\,dx \\
& = \int_{Q_1} \Bigl|\int_{Q_1}u_j(x)-u_j(x+\xi)+Mx\, d\xi\Bigr|^p\,dx \\
& \leq 2^{p-1} \Bigl\{\int_{Q_1}\int_{Q_1}|u_j(x+\xi)-u_j(x)|^p\, dx\, d\xi + d^{\frac{p}{2}}|M|^p\Bigr\},
\end{align*}
which, combined with \eqref{criticalbound}, leads to
\begin{equation*}
    \sup_j \int_{Q_1} |u_j(x)|^p\,dx \leq 2^{p-1}\{(2\lambda)^{d+p}S+d^{\frac{p}{2}}|M|^p\}<+\infty,
\end{equation*}
that implies the thesis.
\end{proof}

For the remaining part of this section, we let $\{\ej\}_j$ and $\{\dej\}_j$ be positive sequences such that $\ej\to 0^+$ and $\dej\to0^+$ as $j\to+\infty$, we set
\begin{equation*}
    \lambda_j:=\frac{\ej}{\dej}, \quad j\in\N,
\end{equation*}
and we assume there exists
\begin{equation*}
\lambda:=\lim_{j\to+\infty}\lambda_j\in[0,+\infty].
\end{equation*}
We let $T$ be fixed so that
\begin{equation*}
    T>r_0 \quad \text{ and } \quad Q_1\subset B_T,
\end{equation*}
with $r_0$ the positive constant appearing in \eqref{rho : ball}; both these conditions are not restrictive since, eventually, we will let $T\to+\infty$.

We consider the truncated functionals 
\begin{equation*}
 F^T_j(u,A)  := \int_{B_T} \rho(\xi)\int_{A_{\ej}(\xi)} f\Bigl(\frac{x}{\dej},\frac{x+\ej\xi}{\dej},\frac{u(x+\ej\xi)-u(x)}{\ej}\Bigr)\, dx\,d\xi   
\end{equation*}
 for $u\in L^p(\Omega; \R^m)$ and $A\in \A_{\rm reg}(\Omega)$, and we recall that the kernel $\rho$ satisfies the assumptions \eqref{rho : ball} and \eqref{rho : moment}, and that the density $f$ satisfies the assumptions \eqref{f : period}, \eqref{f : convex}, and \eqref{f : growth}. We prove that if the $\Gamma$-limit of the truncated functionals $\{F_j^T\}_j$ exists and equals an integral functional, then its energy density $\f$ can be characterized through an `asymptotic formula' and, when $\lambda$ is finite, also by means of certain `non-local cell-problem formulas' involving the parameters $\{\lambda_j\}_j$. 

\begin{proposition}\label{prop : cell formula}
    Assume there exists a quasiconvex function $\varphi: \R^{m\times d}\to[0,+\infty)$ such that 
\begin{equation*}
     \Gamma(L^p)\text{-}\lim_{j\to+\infty} F^T_j(u,A) = \begin{cases}
            \displaystyle \int_{A} \f(\nabla u)\,dx & \text{ if } u\in W^{1,p}(A;\R^m), \\
            +\infty & \text{ if } u\in L^p(\Omega;\R^m)\setminus W^{1,p}(A;\R^m)
        \end{cases}
\end{equation*}
for every $A\in\A_{\rm reg}(\Omega)$. For every $M\in\R^{m\times d}$ the following hold:
\begin{itemize}
    \item[(i)]  for every $\lambda\in[0,+\infty]$ we have
\end{itemize}\begin{align*}
    \f(M) & = \lim_{j\to+\infty} \inf\Bigl\{ \Bigl(\frac{\dej}{r}\Bigr)^d\int_{B_T} \rho(\xi)\int_{(Q_{r/\dej})_{\lambda_j}(\xi)} f\Bigl(x, x+\lambda_j\xi, \frac{u(x+\lambda_j\xi)-u(x)}{\lambda_j}\Bigr)\, dx\,d\xi : \\
    & \qquad \qquad \qquad u\in \mathcal{D}_{T\lambda_j,M}(Q_{r/\dej};\R^m) \Bigr\}
\end{align*}
for some $r>0$;
\item[(ii)] for every $\lambda\in[0,+\infty]$ we have 
\begin{equation*}
    \f(M) \geq \limsup_{j\to+\infty} \inf\Bigl\{ \int_{B_T} \rho(\xi)\int_{Q_1} f\Bigl(x,x+\lambda_j\xi,\frac{u(x+\lambda_j\xi)-u(x)}{\lambda_j}\Bigr)\, dx\,d\xi : u\in L^p_{\#,M}(Q_1;\R^m) \Bigr\};
\end{equation*}
\item[(iii)] for every $\lambda\in[0,+\infty)$ we have
\begin{equation*}
    \f(M) = \lim_{j\to+\infty} \inf\Bigl\{ \int_{B_T} \rho(\xi)\int_{Q_1} f\Bigl(x,x+\lambda_j\xi,\frac{u(x+\lambda_j\xi)-u(x)}{\lambda_j}\Bigr)\, dx\,d\xi : u\in L^p_{\#,M}(Q_1;\R^m) \Bigr\}.
\end{equation*}
\end{proposition}
\begin{proof}
Fix $M\in \R^{m\times d}$ and let $Q$ be a cube of side-length $r$ which is contained in $\Omega$. Using the quasiconvexity of $\f$ and the convergence of boundary-value problems established in Proposition \ref{prop : boundary} with $A=Q$ and $s=T$, we have
    \begin{align*}
      \f(M) & = \inf \Bigl\{ \frac{1}{r^d}\int_{Q} \f(\nabla u) \, dx: u-Mx\in W^{1,p}_0(Q; \R^m) \Bigr\} \\ 
      & = \lim_{j\to+\infty} \inf\Bigl\{ \frac{1}{r^d}F^T_j(u, Q) : u\in \mathcal{D}_{T\ej,M}(Q; \R^m) \Bigr\}.
    \end{align*}
For the sake of exposition, we assume that $Q=Q_r=(0,r)^d$, the general case being analogous. Consider $u\in \mathcal{D}_{T\ej,M}(Q_r;\R^m)$, and set $v(x):=u(\dej x)/\dej$. We have that $v\in \mathcal{D}_{T\lambda_j,M}(Q_{r/\dej}; \R^m)$ and, by a change of variables, 
\begin{align*}
    F_j^T(u, Q_r) & = \int_{B_T} \rho(\xi)\int_{(Q_r)_{\ej}(\xi)} f\Bigl(\frac{x}{\dej},\frac{x+\ej\xi}{\dej},\frac{u(x+\ej\xi)-u(x)}{\ej}\Bigr)\, dx\,d\xi \\
    & = \int_{B_T} \rho(\xi)\int_{(Q_r)_{\ej}(\xi)} f\Bigl(\frac{x}{\dej},\frac{x+\ej\xi}{\dej},\frac{v(\tfrac{x}{\dej}+\tfrac{\ej}{\dej}\xi)-v(\tfrac{x}{\dej})}{\ej/\dej}\Bigr)\, dx\,d\xi \\
    & = \dej^d\int_{B_T} \rho(\xi)\int_{(Q_{r/\dej})_{\lambda_j}(\xi)} f\Bigl(x, x+\lambda_j\xi, \frac{v(x+\lambda_j\xi)-v(x)}{\lambda_j}\Bigr)\, dx\,d\xi.
\end{align*}
By the arbitrariness of $u$, we obtain that 
\begin{align*}
    \f(M) & = \lim_{j\to+\infty} \inf\Bigl\{ \Bigl(\frac{\dej}{r}\Bigr)^d\int_{B_T} \rho(\xi)\int_{(Q_{r/\dej})_{\lambda_j}(\xi)} f\Bigl(x, x+\lambda_j\xi, \frac{u(x+\lambda_j\xi)-u(x)}{\lambda_j}\Bigr)\, dx\,d\xi : \\
    & \qquad \qquad \qquad u\in \mathcal{D}_{T\lambda_j,M}(Q_{r/\dej}; \R^m) \Bigr\},
\end{align*}
which proves $(i)$.

\smallskip

Now we let
\begin{equation*}
    \f_j(M):=\inf\Bigl\{ \int_{B_T} \rho(\xi)\int_{Q_1} f\Bigl(x,x+\lambda_j\xi,\frac{u(x+\lambda_j\xi)-u(x)}{\lambda_j}\Bigr)\, dx\,d\xi : u\in L^p_{\#,M}(Q_1; \R^m) \Bigr\}
\end{equation*}
and prove that
\begin{equation}\label{cell lemma}
    \f(M)\geq \limsup_{j\to+\infty} \f_j(M).
\end{equation}

Given $u\in \mathcal{D}_{T\lambda_j,M}(Q_{r/\dej}; \R^m)$, we define
\begin{equation*}
    \widetilde{u}(x):=
    \begin{cases}
        u(x) & \text{ if } x\in Q_{r/\dej}, \\
        Mx & \text{ if } x\in Q_{\lceil r/\dej\rceil}\setminus Q_{r/\dej},
    \end{cases}
\end{equation*}
we note that $\widetilde{u}\in \mathcal{D}_{T\lambda_j,M}(Q_{\lceil r/\dej\rceil}; \R^m)$, and we write
\begin{align*}
    &\Bigl(\frac{\dej}{r}\Bigr)^d\int_{B_T} \rho(\xi)\int_{(Q_{\lceil r/\dej \rceil })_{\lambda_j}(\xi)} f\Bigl(x, x+\lambda_j\xi, \frac{\widetilde{u}(x+\lambda_j\xi)-\widetilde{u}(x)}{\lambda_j}\Bigr)\, dx\,d\xi \\
    & = \Bigl(\frac{\dej}{r}\Bigr)^d\int_{B_T} \rho(\xi)\int_{(Q_{r/\dej})_{\lambda_j}(\xi)} f\Bigl(x, x+\lambda_j\xi, \frac{u(x+\lambda_j\xi)-u(x)}{\lambda_j}\Bigr)\, dx\,d\xi \\
    & \quad + \Bigl(\frac{\dej}{r}\Bigr)^d\int_{B_T} \rho(\xi)\int_{(Q_{\lceil r/\dej \rceil })_{\lambda_j}(\xi) \setminus (Q_{r/\dej})_{\lambda_j}(\xi)} f\Bigl(x, x+\lambda_j\xi, \frac{\widetilde{u}(x+\lambda_j\xi)-\widetilde{u}(x)}{\lambda_j}\Bigr)\, dx\,d\xi.
\end{align*}
If $\xi\in B_T$ and $x\in (Q_{\lceil r/\dej \rceil })_{\lambda_j}(\xi) \setminus (Q_{r/\dej})_{\lambda_j}(\xi)$, then $\widetilde{u}(x+\lambda_j\xi)-\widetilde{u}(x)=\lambda_j M\xi$, hence, using the upper bound in \eqref{f : growth} we get
\begin{align} \nonumber
 & \Bigl(\frac{\dej}{r}\Bigr)^d\int_{B_T} \rho(\xi)\int_{(Q_{\lceil r/\dej \rceil })_{\lambda_j}(\xi) \setminus (Q_{r/\dej})_{\lambda_j}(\xi)} f\Bigl(x, x+\lambda_j\xi, \frac{\widetilde{u}(x+\lambda_j\xi)-\widetilde{u}(x)}{\lambda_j}\Bigr)\, dx\,d\xi \\ \nonumber
    & = \Bigl(\frac{\dej}{r}\Bigr)^d\int_{B_T} \rho(\xi)\int_{(Q_{\lceil r/\dej \rceil })_{\lambda_j}(\xi) \setminus (Q_{r/\dej})_{\lambda_j}(\xi)} f(x, x+\lambda_j\xi, M\xi)\, dx\,d\xi \\ \nonumber
    & \leq  \Bigl(\frac{\dej}{r}\Bigr)^d|E_j|\beta|M|^p\int_{B_T} \rho(\xi)|\xi|^p \, d\xi \\ \nonumber
    & =: \theta_j,
\end{align}
where we have set
\begin{equation}\label{set}
    E_j:=Q_{\lceil r/\dej \rceil } \setminus Q_{r/\dej-T\lambda_j}, \quad j\in\N,
\end{equation}
and we have used that $(Q_{\lceil r/\dej \rceil })_{\lambda_j}(\xi) \setminus (Q_{r/\dej})_{\lambda_j}(\xi)\subset E_j$ for all $\xi\in B_T$. Recalling \eqref{rho : moment} we infer there exists a positive constant $C$, independent of $j$, such that
\begin{equation*}
 \theta_j  = C \Bigl(\frac{\dej}{r}\Bigr)^d \Bigl( \Bigl\lceil \frac{r}{\dej} \Bigr\rceil^d - \Bigl(\frac{r}{\dej}-T\lambda_j\Bigr)^d \Bigr)  = C\Bigl[\Bigl( \frac{\lceil {r/\dej} \rceil}{r/\dej}\Bigr)^d - \Bigl(1-\frac{T}{r}\ej\Bigr)^d \Bigr],
\end{equation*}
which tends to $0$ as $j\to+\infty$. As a consequence of these observations, we get that
\begin{align*}
     &\Bigl(\frac{\dej}{r}\Bigr)^d\int_{B_T} \rho(\xi)\int_{(Q_{r/\dej})_{\lambda_j}(\xi)} f\Bigl(x,x+\lambda_j\xi,\frac{u(x+\lambda_j\xi)-u(x)}{\lambda_j}\Bigr)\, dx\,d\xi \\
     & \geq \Bigl(\frac{\dej}{r}\Bigr)^d\int_{B_T} \rho(\xi)\int_{(Q_{\lceil r/\dej \rceil })_{\lambda_j}(\xi)} f\Bigl(x,x+\lambda_j\xi,\frac{\widetilde{u}(x+\lambda_j\xi)-\widetilde{u}(x)}{\lambda_j}\Bigr)\, dx\,d\xi -\theta_j,    
\end{align*}
which, by the arbitrariness of $u$, yields
\begin{align*}
    \f(M) & = \lim_{j\to+\infty} \inf\Bigl\{ \Bigl(\frac{\dej}{r}\Bigr)^d\int_{B_T} \rho(\xi)\int_{(Q_{r/\dej})_{\lambda_j}(\xi)} f\Bigl(x,x+\lambda_j\xi,\frac{u(x+\lambda_j\xi)-u(x)}{\lambda_j}\Bigr)\, dx\,d\xi : \\
    & \qquad \qquad \qquad  u\in \mathcal{D}_{T\lambda_j,M}(Q_{r/\dej}; \R^m) \Bigr\} \\
    & \geq \limsup_{j\to+\infty} \inf\Bigl\{\Bigl(\frac{\dej}{r}\Bigr)^d\int_{B_T} \rho(\xi)\int_{(Q_{\lceil r/\dej\rceil})_{\lambda_j}(\xi)} f\Bigl(x,x+\lambda_j\xi,\frac{u(x+\lambda_j\xi)-u(x)}{\lambda_j}\Bigr)\, dx\,d\xi : \\
    & \qquad \qquad \qquad \quad u\in \mathcal{D}_{T\lambda_j,M}(Q_{\lceil r/\dej\rceil}; \R^m) \Bigr\}.
\end{align*}

Consider now $u\in \mathcal{D}_{T\lambda_j,M}(Q_{\lceil r/\dej\rceil};\R^m)$, let $\widetilde{w}$ be the $Q_{\lceil r/\dej\rceil}$-periodic extension of $u(x)-Mx$, let $\widetilde{u}(x):=\widetilde{w}(x)+Mx$, and define the function 
\begin{equation*}
    w(x):=\frac{1}{\lceil r/\dej \rceil^d} \sum_{i\in \mathbb{Z}^d\cap [0,\lceil r/\dej \rceil)^d} \widetilde{u}(x+i).
\end{equation*}
Since $w\in L^p_{\#,M}(Q_1; \R^m)$, by the assumptions \eqref{f : convex} and \eqref{f : period} we obtain
\begin{align} \nonumber
    \f_j(M) & \leq \int_{B_T} \rho(\xi)\int_{Q_1} f\Bigl(x,x+\lambda_j\xi,\frac{w(x+\lambda_j\xi)-w(x)}{\lambda_j}\Bigr)\, dx\,d\xi \\ \nonumber
    & \leq \frac{1}{\lceil r/\dej \rceil^d} \sum_{i\in \mathbb{Z}^d\cap [0,\lceil r/\dej \rceil)^d}  \int_{B_T} \rho(\xi)\int_{Q_1} f\Bigl(x, x+\lambda_j\xi, \frac{\widetilde{u}(x+i+\lambda_j\xi)-\widetilde{u}(x+i)}{\lambda_j}\Bigr)\, dx\,d\xi \\ \nonumber
    & = \frac{1}{\lceil r/\dej \rceil^d} \sum_{i\in \mathbb{Z}^d\cap [0,\lceil r/\dej \rceil)^d}  \int_{B_T} \rho(\xi)\int_{i+Q_1} f\Bigl(x, x+\lambda_j\xi, \frac{\widetilde{u}(x+\lambda_j\xi)-\widetilde{u}(x)}{\lambda_j}\Bigr)\, dx\,d\xi \\ \label{cell 1}
    & = \frac{1}{\lceil r/\dej \rceil^d} \int_{B_T} \rho(\xi)\int_{Q_{\lceil r/\dej\rceil}} f\Bigl(x,x+\lambda_j\xi,\frac{u(x+\lambda_j\xi)-u(x)}{\lambda_j}\Bigr)\, dx\,d\xi.
\end{align}
Moreover, observing that $Q_{\lceil r/\dej\rceil}\setminus (Q_{\lceil r/\dej\rceil})_{\lambda_j}(\xi)\subset E_j$ for all $\xi\in B_T$ and using \eqref{f : growth}, we get 
\begin{align} \nonumber
    & \frac{1}{\lceil r/\dej \rceil^d} \int_{B_T} \rho(\xi)\int_{Q_{\lceil r/\dej\rceil}\setminus (Q_{\lceil r/\dej\rceil})_{\lambda_j}(\xi)} f\Bigl(x, x+\lambda_j\xi,\frac{u(x+\lambda_j\xi)-u(x)}{\lambda_j}\Bigr)\, dx\,d\xi \\ \nonumber
    & \leq \frac{1}{\lceil r/\dej \rceil^d} |E_j| \beta |M|^p\int_{B_T} \rho(\xi)|\xi|^p\, d\xi \\ \label{cell 2}
    & = \frac{(r/\dej)^d}{\lceil r/\dej \rceil^d}  \theta_j.
 \end{align}
Therefore, combining \eqref{cell 1} with \eqref{cell 2}, we obtain
\begin{align*}
  \Bigl(\frac{\dej}{r}\Bigr)^d\int_{B_T} \rho(\xi)\int_{(Q_{\lceil r/\dej\rceil})_{\lambda_j}(\xi)} f\Bigl(x,x+\lambda_j\xi,\frac{u(x+\lambda_j\xi)-u(x)}{\lambda_j}\Bigr)\, dx\,d\xi \geq \Bigl(\frac{\lceil r/\dej \rceil}{r/\dej}\Bigr)^d\f_j(M)-\theta_j;
\end{align*}
hence, taking into account the arbitrariness of $u$ and that $\theta_j\to0$, we infer \eqref{cell lemma} passing to the limit, concluding the proof of $(ii)$.

\smallskip

Finally, we assume that  $\lambda\in[0,+\infty)$ and  we prove that
\begin{equation}\label{cell lemma bis}
      \f(M) \leq \liminf_{j\to+\infty} \f_j(M).
\end{equation}
Let $\{u_j\}_j\subset L^p_{\#,M}(Q_1;\R^m)$ be such that 
\begin{equation*}
    \liminf_{j\to+\infty} \f_j(M) = \liminf_{j\to+\infty} \int_{B_T} \rho(\xi)\int_{Q_1} f\Bigl(x,x+\lambda_j\xi,\frac{u_j(x+\lambda_j\xi)-u_j(x)}{\lambda_j}\Bigr)\, dx\,d\xi, 
\end{equation*}
and, by the translation invariance of the functional, suppose that
\begin{equation*}
    \int_{Q_1}u_j\,dx=0
\end{equation*}
for every $j\in\N$. Since $T>r_0$, by \eqref{rho : ball}, the lower bound in \eqref{f : growth}, and the previous step, we have 
\begin{equation*}
\f(M)\geq \alpha c_0\liminf_{j\to+\infty} \int_{B_{r_0}} \int_{Q_1}\Bigl|\frac{u_j(x+\lambda_j\xi)-u_j(x)}{\lambda_j}\Bigr|^p \, dx\,d\xi;
\end{equation*}
therefore, upon extracting a not relabeled subsequence,  we apply Lemma \ref{lemma : compactness} with $r=r_0$ and $A=Q_1$ to obtain that $\sup_{j} \|u_j\|_{L^p(Q_1;\R^m)}<+\infty$, and then
\begin{equation}\label{uniform bound}
 \sup_{j} \|u_j-Mx\|_{L^p(Q_1;\R^m)}<+\infty.
\end{equation}
Now we set $w_j(x):=\dej u_j(x/\dej), j\in \N$, and note that $w_j\to Mx$ in $L^p(Q_1;\R^m)$ as $j\to+\infty$. To see this, we also let $v_j(x):=u_j(x)-Mx, j\in\N$, and observe that, since $v_j$ is $Q_1$-periodic, \begin{align*}
    \int_{Q_1}|w_j(x)-Mx|^p\,dx & = \int_{Q_1}\Bigl|\dej v_j\Bigl(\frac{x}{\dej}\Bigr)\Bigr|^p\,dx \\
    & = \dej^{d+p}\int_{Q_{1/\dej}} |v_j(x)|^p\, dx \\
    & \leq \frac{\lceil1/\dej\rceil^d} {(1/\dej)^d}\dej^p\int_{Q_1}|v_j(x)|^p\,dx,
\end{align*}
which tends to $0$ by \eqref{uniform bound} and the fact that $\dej\to0$. Using the definition of $\Gamma$-limit and a change of variables, we obtain
\begin{align*}
  \f(M) = \int_{Q_1} \f(M)\, dx & \leq \liminf_{j\to+\infty} F^T_j(w_j, Q_1) \\
    & = \liminf_{j\to+\infty} \int_{B_T} \rho(\xi)\int_{{(Q_1)}_{\ej}(\xi)} f\Bigl(\frac{x}{\dej},\frac{x+\ej\xi}{\dej},\frac{w_j(x+\ej\xi)-w_j(x)}{\ej}\Bigr)\, dx\,d\xi \\
    & = \liminf_{j\to+\infty} \dej^d \int_{B_T} \rho(\xi)\int_{(Q_{1/\dej})_{\lambda_j}(\xi)} f\Bigl(x,x+\lambda_j\xi,\frac{u_j(x+\lambda_j\xi)-u_j(x)}{\lambda_j}\Bigr)\, dx\,d\xi \\
    & \leq \liminf_{j\to+\infty} \dej^d \int_{B_T} \rho(\xi)\int_{Q_{\lceil 1/\dej \rceil}} f\Bigl(x,x+\lambda_j\xi,\frac{u_j(x+\lambda_j\xi)-u_j(x)}{\lambda_j}\Bigr)\, dx\,d\xi \\
    & =  \liminf_{j\to+\infty} \int_{B_T} \rho(\xi)\int_{Q_1} f\Bigl(x,x+\lambda_j\xi,\frac{u_j(x+\lambda_j\xi)-u_j(x)}{\lambda_j}\Bigr)\, dx\,d\xi \\
    & = \liminf_{j\to+\infty}\f_j(M),
     \end{align*}
where we also used \eqref{f : period} and the $Q_1$-periodicity of the functions
\begin{equation*}
    x\mapsto \frac{u_j(x+\lambda_j\xi)-u_j(x)}{\lambda_j}, \quad j\in\N.
\end{equation*}
This leads to \eqref{cell lemma bis} and proves $(iii)$.
\end{proof}

Now we explicitly determine the integrand $\f$ for each value of $\lambda\in[0,+\infty]$. We devote the next subsections to the proof of the following proposition.

\begin{proposition}\label{prop : main}  Let $r_0$ be the positive constant appearing in \eqref{rho : ball} and let $T$ be such that $T>r_0$ and $Q_1\subset B_T$. Assume that there exist
\begin{equation*}
   \lambda= \lim_{j\to+\infty}\frac{\ej}{\dej}\in[0,+\infty]
\end{equation*}
and a quasiconvex function $\varphi: \R^{m\times d}\to[0,+\infty)$ such that 
\begin{equation}\label{ipotesi rappresentazione}
     \Gamma(L^p)\text{-}\lim_{j\to+\infty} F^T_j(u,A) = \int_A \f(\nabla u)\, dx, \quad u\in W^{1,p}(\Omega;\R^m),
\end{equation}
and $+\infty$ otherwise in $L^p(\Omega;\R^m)$, for every $A\in\A_{\rm reg}(\Omega)$. Then the following hold for every $M\in \R^{m\times d}$: 
 \begin{itemize}
        \item[(i)] if $\lambda=0$ and \eqref{f : continuous} holds, then
        \begin{equation*}
            \f(M)= \inf\Bigl\{ \int_{B_T} \int_{Q_1} \rho(\xi)f(x,x,(\nabla u ) \xi)\, dx\,d\xi : u\in W^{1,p}_{\#,M}(Q_1; \R^m) \Bigr\};
        \end{equation*}
        \item[(ii)] if $\lambda\in(0,+\infty)$ and \eqref{f : continuous} holds, then 
         \begin{equation*}
            \f(M)= \inf\Bigl\{ \int_{B_T}\int_{Q_1} \rho(\xi) f\Bigl(x,x+\lambda\xi,\frac{u(x+\lambda\xi)-u(x)}{\lambda}\Bigr)\, dx\, d\xi : u\in L^p_{\#,M}(Q_1;\R^m) \Bigr\};
        \end{equation*}
        \item[(iii)] if $\lambda=+\infty$ and either \eqref{H1} or \eqref{H2} holds, then
        \begin{equation*}
            \f(M)=\int_{B_T} \int_{Q_1}\int_{Q_1} \rho(\xi) f(x,y,M\xi)\, dx\,dy\,d\xi.
        \end{equation*}
    \end{itemize}
\end{proposition}

\subsection{The subcritical and critical cases, $\lambda\in[0,+\infty)$}

The cases $\lambda=0$ and $\lambda\in(0,+\infty)$ are treated with similar arguments. We recall that, in both these instances, we assume that
\begin{equation*} 
    \tag{H0}
f(x,\cdot,z) \text{ is continuous for almost every } x\in \R^d \text{ and for every } z\in \R^m.
\end{equation*}
 For the sake of exposition, we first establish the corresponding lower bounds, and then we prove their optimality.
 
\subsubsection{Lower bounds}

Let $M\in \R^{m\times d}$ be fixed. By \eqref{ipotesi rappresentazione} and $(iii)$ of Proposition \ref{prop : cell formula}, there exists $\{u_j\}_j\subset L^p_{\#,M}(Q_1; \R^m)$ such that
\begin{equation*}
\f(M)=\lim_{j\to+\infty} \int_{B_T} \rho(\xi)\int_{Q_1} f\Bigl(x,x+\lambda_j\xi,\frac{u_j(x+\lambda_j\xi)-u_j(x)}{\lambda_j}\Bigr)\, dx\,d\xi ,
\end{equation*}
and which, by the translation invariance of the functional, it is not restrictive to suppose satisfy
\begin{equation}\label{mean value bis}
    \int_{Q_1} u_j\, dx = 0
\end{equation}
for every $j\in\N$. Using the lower bound in \eqref{f : growth} and \eqref{rho : ball}, we get that
\begin{equation*}
\sup \int_{B_{r_0}} \int_{Q_1}\Bigl|\frac{u_j(x+\lambda_j\xi)-u_j(x)}{\lambda_j}\Bigr|^p \, dx\,d\xi<+\infty,
\end{equation*}
and then, by Lemma \ref{lemma : diadic} applied with $r=r_0$ and $E=B_T$, we obtain
\begin{equation}\label{preliminarybound}
 \sup_{j} \int_{B_{T}} \int_{Q_1}\Bigl|\frac{u_j(x+\lambda_j\xi)-u_j(x)}{\lambda_j}\Bigr|^p \, dx\,d\xi<+\infty.
\end{equation}
Therefore, we set
\begin{equation*}
    U_j(x,\xi):=\frac{u_j(x+\lambda_j\xi)-u_j(x)}{\lambda_j}, \quad j\in\N,
\end{equation*}
and infer that the sequence of functions $\{U_j\}_j$ is equi-bounded in $L^p(Q_1\times B_T; \R^m)$. Upon extracting a (not relabeled) converging subsequence, we may assume there exists $U$ such that 
\begin{equation}\label{weaksub}
     U_j(x,\xi) \rightharpoonup U(x,\xi) \quad \text{ as } j\to+\infty
\end{equation}
weakly in $L^p(Q_1\times B_T; \R^m)$; then, we apply Theorem \ref{thm : LSC} with $\ell=d\times d, s=(\xi,x)$, and
\begin{equation*}
   E=B_T\times Q_1,  \quad  \Psi((\xi,x),t,q)=\rho(\xi)f(x,t,q),
\end{equation*}
to infer
\begin{align} \nonumber
\f(M)& =\lim_{j\to+\infty} \int_{B_T} \rho(\xi)\int_{Q_1} f\Bigl(x,x+\lambda_j\xi,\frac{u_j(x+\lambda_j\xi)-u_j(x)}{\lambda_j}\Bigr)\, dx\,d\xi \\ \label{lowerboundgenerale}
& \geq \int_{B_T} \rho(\xi)\int_{Q_1} f(x, x+\lambda\xi, U(x,\xi))\, dx\,d\xi.
\end{align}

At this point, we identify the weak limit $U$ in accordance with the value of the parameter $\lambda$ and obtain different lower bounds.

\medskip

{\em Subcritical case, $\lambda=0$.} By the $Q_1$-periodicity of $u_j-Mx$, it easy to observe that \eqref{preliminarybound} implies  
\begin{equation}\label{preliminarybound1}
 \sup_{j} \int_{B_{T}} \int_{A}\Bigl|\frac{u_j(x+\lambda_j\xi)-u_j(x)}{\lambda_j}\Bigr|^p \, dx\,d\xi<+\infty
\end{equation}
for $A$ any bounded open set. Resorting to \eqref{mean value bis} and \eqref{preliminarybound1}, we apply $(i)$ of Lemma \ref{lemma : compactness} with $A$ any bounded open set containing the closure of $Q_1$ and $r=T$ in order to obtain a (not relabeled) subsequence $\{u_{j}\}_j$ and a function $u\in W_{\#,M}^{1,p}(Q_1;\R^m)$ such that $u_{j}\to u$ strongly in $L^p(Q_1;\R^m)$. Note that the direct use of $(i)$ of Lemma \ref{lemma : compactness} with $A=Q_1$ would only imply $u-Mx\in W^{1,p}(Q_1;\R^m)$, which, extended by periodicity on $\R^d$, does not necessarily belong to $ W^{1,p}_{\text{loc}}(\R^d;\R^m)$. 

We claim that 
\begin{equation}\label{subclaim}
   U(x,\xi)=(\nabla u(x))\xi
\end{equation}
for a.e.\ $x\in Q_1$ and $\xi\in B_T$. By the equi-boundedness of $\{U_j\}_j$, it is sufficient testing the weak convergence stated in \eqref{weaksub} with a class of functions whose linear span is dense in the dual space of $L^p(Q_1\times B_T;\R^m)$. Let then $\psi\in C^\infty_c(Q_1;\R^m)$ and $E\subseteq B_T$ be measurable; by a change of variables we have
\begin{align*}
\int_E\int_{Q_1}\psi(x)\cdot U_j(x,\xi)\, dx \, d\xi & = \int_E\int_{Q_1}\psi(x)\cdot\frac{u_j(x+\lambda_j\xi)-u_j(x)}{\lambda_j}\, dx \, d\xi \\ 
& = \int_E \int_{\R^d}u_j(x)\cdot\frac{\psi(x-\lambda_j\xi)-\psi(x)}{\lambda_j}\, dx\, d\xi.
\end{align*}
We observe that, since $u_j-Mx$ is $Q_1$-periodic and $u_j\to u$ strongly in $L^p(Q_1;\R^m)$, we have that $u_j\to u$ strongly in $L^p(K;\R^m)$ for $K$ any compact subset of $\R^d$; therefore, recalling that $\psi$ has compact support, it is possible to pass to the limit using the Dominated Convergence Theorem to get
\begin{align*}
    \int_E\int_{Q_1}\psi(x)\cdot U(x,\xi)\, dx \, d\xi & = -\int_E \int_{\R^d}u(x)\cdot(\nabla \psi(x))\xi\, dx\, d\xi \\
    & = \int_E \int_{Q_1}\psi(x)\cdot(\nabla u(x))\xi\, dx\, d\xi,
\end{align*}
which proves \eqref{subclaim}. Recalling \eqref{lowerboundgenerale}, we get
\begin{equation*}
    \f(M)\geq \inf\Bigl\{ \int_{B_T} \rho(\xi)\int_{Q_1} f(x,x,(\nabla u ) \xi)\, dx\,d\xi : u\in W^{1,p}_{\#,M}(Q_1; \R^m) \Bigr\},
\end{equation*}
that is the lower bound for the case $\lambda=0$.

\medskip

{\em Critical case, $\lambda\in(0,+\infty)$.} Using \eqref{mean value bis}, \eqref{preliminarybound}, and $(ii)$ of Lemma \ref{lemma : compactness}, we obtain a (not relabeled) subsequence $\{u_{j}\}_j$ and a function $u\in L^p_{\#,M}(Q_1; \R^m)$ such that $u_{j}\rightharpoonup u$ weakly in $L^p(Q_1; \R^m)$. Using an argument similar to that employed in the subcritical case, it is then immediate to verify that
\begin{equation*}
    U(x,\xi)=\frac{u(x+\lambda\xi)-u(x)}{\lambda},
\end{equation*}
which, combined with \eqref{lowerboundgenerale}, implies
\begin{equation*}
    \f(M)  \geq \inf\Bigl\{ \int_{B_T} \rho(\xi)\int_{Q_1} f\Bigl(x,x+\lambda\xi,\frac{u(x+\lambda\xi)-u(x)}{\lambda}\Bigr)\, dx\, d\xi : u\in L^p_{\#,M}(Q_1;\R^m) \Bigr\},
\end{equation*}
achieving the lower bound in the case $\lambda\in(0,+\infty)$.

\subsubsection{Upper bounds}

Now we prove that the lower bounds previously obtained are optimal. Although the proofs for each regime are similar, we illustrate them separately for the sake of clarity.

\medskip

{\em Subcritical case, $\lambda=0$.} Let $\eta>0$ and let $w\in W_{\#,M}^{1,p}(Q_1; \R^m)$ be such that
    \begin{equation*}
      \inf\Bigl\{ \int_{B_T} \rho(\xi)\int_{Q_1} f(x,x,(\nabla u)\xi)\, dx\,d\xi : u\in W^{1,p}_{\#,M}(Q_1; \R^m) \Bigr\} + \eta \geq  \int_{B_T} \rho(\xi)\int_{Q_1} f(x,x,(\nabla w)\xi)\, dx\,d\xi.
    \end{equation*} 

We observe that the functional that appears in the minimum problem is continuous with respect to the strong convergence in $W_{\rm loc}^{1,p}(\R^d;\R^m)$. This follows applying the Dominated Convergence Theorem whose use is justified by \eqref{f : convex}, by the estimate
    \begin{equation}\label{forDCT}
       \rho(\xi) f(x,x,(\nabla u(x))\xi) \leq \beta \rho(\xi)|\xi|^p|\nabla u(x)|^p \quad \text{ for a.e. } (x,\xi)\in Q_1\times B_T,
    \end{equation}
which follows by \eqref{f : growth}, and by \eqref{rho : moment}. Therefore, it is not restrictive to suppose that $w\in C^\infty(\R^d;\R^m)$.

Applying $(iii)$ of Proposition \ref{prop : cell formula} we have
\begin{align*}
    \f(M) & = \lim_{j\to+\infty} \inf\Bigl\{ \int_{B_T} \rho(\xi)\int_{Q_1} f\Bigl(x,x+\lambda_j\xi,\frac{u(x+\lambda_j\xi)-u(x)}{\lambda_j}\Bigr)\, dx\,d\xi : u\in L^p_{\#,M}(Q_1;\R^m) \Bigr\} \\
    & \leq \lim_{j\to+\infty} \int_{B_T} \rho(\xi)\int_{Q_1} f\Bigl(x,x+\lambda_j\xi,\frac{w(x+\lambda_j\xi)-w(x)}{\lambda_j}\Bigr)\, dx\,d\xi.
\end{align*}
In order to conclude, we use the Dominated Convergence Theorem. The pointwise convergence is ensured by the fact that $f$ satisfies \eqref{f : convex} and \eqref{f : continuous}, and that
\begin{equation*}
    \frac{w(x+\lambda_j\xi)-w(x)}{\lambda_j} \to (\nabla w(x)) \xi
\end{equation*}
for all $x,\xi\in\R^d$ as $j\to+\infty$; a uniform bound for the integrands 
\begin{equation*}
    \rho(\xi) f\Bigl(x,x+\lambda_j\xi,\frac{w(x+\lambda_j\xi)-w(x)}{\lambda_j}\Bigr), \quad j\in\N,
\end{equation*}
holds in virtue of \eqref{f : growth}, the inequality
\begin{equation}\label{quotient}
    \Bigl|\frac{w(x+\lambda_j\xi)-w(x)}{\lambda_j}\Bigr| \leq \|\nabla w\|_{L^\infty(\R^d;\R^{m\times d})}|\xi|
\end{equation}
for all $x,\xi\in\R^d$, and \eqref{rho : moment}. Finally, we get
\begin{align*}
     \f(M)  & \leq \int_{B_T} \rho(\xi)\int_{Q_1} f(x,x,(\nabla w(x)) \xi)\, dx\,d\xi \\
     & \leq \Bigl(\inf\Bigl\{ \int_{B_T} \rho(\xi)\int_{Q_1} f(x,x,(\nabla u) \xi)\, dx\,d\xi : u\in W^{1,p}_{\#,M}(Q_1; \R^m) \Bigr\}  +\eta\Bigr),
\end{align*}
     which, by the arbitrariness of $\eta$, concludes the proof.

\medskip

{\em Critical case, $\lambda\in(0,+\infty)$.} Let $\eta>0$ and let $w\in L^p_{\#,M}(Q_1; \R^m)\cap C^\infty(\R^d;\R^m)$ be such that
    \begin{multline*}
      \inf\Bigl\{ \int_{B_T} \rho(\xi)\int_{Q_1} f\Bigl(x,x+\lambda\xi,\frac{u(x+\lambda\xi)-u(x)}{\lambda}\Bigr)\, dx\, d\xi : u\in L^p_{\#,M}(Q_1;\R^m) \Bigr\} +\eta \\
      \geq \int_{B_T} \rho(\xi)\int_{Q_1} f\Bigl(x,x+\lambda\xi,\frac{w(x+\lambda\xi)-w(x)}{\lambda}\Bigr)\, dx\, d\xi.
    \end{multline*} 
Note indeed that a similar argument to that used for $\lambda=0$ proves the continuity of the functional with respect to the strong $L^p$-convergence.
  
  Once again, we apply $(iii)$ of Proposition \ref{prop : cell formula}, use $w$ as a test function for the corresponding minimum problems, and apply the Dominated Convergence Theorem (whose use is made possible also by \eqref{f : convex}, \eqref{f : continuous}, \eqref{f : growth}, \eqref{quotient}, and \eqref{rho : moment}) to infer
\begin{align*}
     \f(M)  & \leq  \int_{B_T} \rho(\xi)\int_{Q_1} f\Bigl(x,x+\lambda\xi,\frac{w(x+\lambda\xi)-w(x)}{\lambda} \Bigl)\, dx\,d\xi \\
     & \leq  \Bigl(\inf\Bigl\{ \int_{B_T} \rho(\xi)\int_{Q_1} f\Bigl(x,x+\lambda\xi,\frac{u(x+\lambda\xi)-u(x)}{\lambda}\Bigr)\, dx\, d\xi : u\in L^p_{\#,M}(Q_1;\R^m) \Bigr\} +\eta\Bigr),
\end{align*}
     which, by the arbitrariness of $\eta$, concludes the proof.

\subsection{The supercritical case, $\lambda=+\infty$.}

The supercritical regime requires a different and more complex argument. Throughout this section, we assume that either
\begin{equation*}
    \tag{H1}
       \rho(\xi)f(x,y,z)=\rho(-\xi)f(x,y,-z) 
    \end{equation*}
    for almost every $\xi,x,y\in \R^d$ and for every $z\in\R^m$, or that 
\begin{equation*}
    \tag{H2}
    \int_{\R^d}\rho(\xi)\,d\xi =+\infty.
\end{equation*}

\subsubsection{Lower bound}

Let $M\in \R^{m\times d}$ be fixed. By $(ii)$ of Proposition \ref{prop : cell formula}, there exists $\{u_j\}_j\subset L^p_{\#,M}(Q_1; \R^m)$ such that
\begin{equation*}
\f(M)\geq \limsup_{j\to+\infty} \int_{B_T} \rho(\xi)\int_{Q_1} f\Bigl(x,x+\lambda_j\xi,\frac{u_j(x+\lambda_j\xi)-u_j(x)}{\lambda_j}\Bigr)\ dx\,d\xi.
\end{equation*}
In the supercritical regime, there are no compactness properties that we are able to infer for the sequence $\{u_j\}_j$. We consider the $Q_1$-periodic functions $v_j(x):=u_j(x)-Mx, j\in\N,$ and we let
\begin{equation*}
    V_j(x,\xi):=\frac{v_j(x+\xi)-v_j(x)}{\lambda_j}, \quad j\in\N,
\end{equation*}
so that
\begin{equation}\label{superstarting}
\f(M)\geq \limsup_{j\to+\infty} \int_{B_T} \rho(\xi)\int_{Q_1} f(x,x+\lambda_j\xi,V_j(x,\lambda_j\xi)+M\xi)\, dx\,d\xi.
\end{equation}

In order to describe the idea behind our proof, we note that if we had $V_j(x,\xi)=V(x,\xi)$ for all $j\in\N$, then
\begin{align*}
\f(M) & \geq \limsup_{j\to+\infty} \int_{B_T} \rho(\xi)\int_{Q_1} f(x,x+\lambda_j\xi,V(x,\lambda_j\xi)+M\xi)\, dx\,d\xi \\
& = \limsup_{j\to+\infty} \int_{B_T} \rho(\xi)g(\xi,\lambda_j\xi)\,d\xi,
\end{align*}
where we set
\begin{equation*}
   g(\xi,y):= \int_{Q_1} f(x,x+y,V(x,y)+M\xi)\, dx.
\end{equation*}
Since $g$ is periodic in the second variable and $\lambda_j\to+\infty$, we may use a known result, which can be regarded as a consequence of a two-scale convergence, to obtain that
\begin{equation*}
    g(\xi, \lambda_j\xi)\rightharpoonup \int_{Q_1}g(\xi,y)\,dy
\end{equation*}
in a weak sense; and then, formally passing to the limit, we would get
\begin{equation*}
\f(M) \geq \int_{B_T}\rho(\xi)\int_{Q_1}\int_{Q_1} f(x,x+y,V(x,y)+M\xi)\, dx\,dy\,d\xi.
\end{equation*}

Our aim is to mimic this reasoning taking into account that $\{V_j\}_j$ is not a constant sequence. This requires the study of the equi-integrability of an auxiliary sequence of functions defined similarly to the function $g$ above, and some additional care has to be used since in general $\rho$ does not belong to any $L^p$-space. These issues seem to prevent the direct use of the two-scale convergence and, also in our setting, allows us to pass to the limit only upon having truncated `vertically' the kernel.

\smallskip

In this subsection, from now on, we shall further assume that
\begin{equation*}
    \inf_j \lambda_j >2;
\end{equation*}
this condition is not restrictive since we are assuming that $\lambda_j\to+\infty$.

We note that each $V_j$ is $Q_1$-periodic in both variables. Moreover, resorting to Lemma \ref{lemma : diadic}, we have that \eqref{preliminarybound} holds, which implies
\begin{equation}\label{superbound}
\sup_{j} \|V_j(x,\lambda_j\xi)\|_{L^p(Q_1\times B_T; \R^m)} <+\infty.
\end{equation}
The following observation is then a simple consequence.

\begin{lemma}\label{lemma : supercritical 1} We have
\begin{equation*}
    \sup_{j} \|V_j(x,\xi)\|_{L^p(Q_1\times Q_1; \R^m)} <+\infty.
\end{equation*}
\end{lemma}
\begin{proof}
    Since $Q_1\subset B_T$, we combine the periodicity of $V_j$ in the second variable with a change of variables to get
\begin{align*}
    \int_{B_T}\int_{Q_1} |V_j(x,\lambda_j\xi)|^p\,dx\,d\xi & \geq \int_{Q_1}\int_{Q_{1}} |V_j(x,\lambda_j\xi)|^p\,dx\,d\xi \\
    & = \lambda_j^{-d}\int_{Q_{\lambda_j}}\int_{Q_{1}} |V_j(x, \xi)|^p\,dx\,d\xi \\
    & \geq \lambda_j^{-d} \lfloor \lambda_j \rfloor^d \int_{Q_1}\int_{Q_{1}} |V_j(x, \xi)|^p\,dx\,d\xi \\
    & \geq \frac{1}{2}\int_{Q_1}\int_{Q_{1}} |V_j(x, \xi)|^p\,dx\,d\xi,
\end{align*}
where in the last inequality we have used that $\lambda_j>2$ for all $j\in\N$. By \eqref{superbound}, the proof is concluded.
\end{proof}

Now we let
\begin{equation*}
    g_j(\xi):= \int_{Q_1} f(x,x+\lambda_j\xi,V_j(x,\lambda_j\xi)+M\xi)\, dx, \quad \xi\in \R^d,
\end{equation*}
\begin{equation*}
    \widetilde{g}_j(\xi):=\int_{Q_1} \int_{Q_1} f(x,x+y,V_j(x,y)+M\xi)\, dx\, dy, \quad \xi\in \R^d,
\end{equation*}
and observe the following.

\begin{lemma}\label{lemma : supercritical 2}
    The sequences $\{g_j\}_j$ and $\{\widetilde{g}_j\}_j$ are equi-integrable in $L^1(B_T)$.
\end{lemma}
\begin{proof}
Consider first the sequence $\{g_j\}_j$. By a density argument, the thesis follows if we prove that there exists a positive constant $C$ such that for any cube $Q\subset B_T$ it holds
\begin{equation*}
  \limsup_{j\to+\infty}  \int_Q g_j(\xi)\,d\xi\leq C|Q|.
\end{equation*} 

To this end, we first observe that by \eqref{f : growth} we have
\begin{align}\notag
\int_Q g_j(\xi)\,d\xi & = \int_Q\int_{Q_1} f(x,x+\lambda_j\xi,V_j(x,\lambda_j\xi)+M\xi)\, dx\,d\xi \\ \notag
& \leq \beta \int_Q\int_{Q_1} |V_j(x,\lambda_j\xi)+M\xi|^p\, dx\,d\xi \\ \label{equi1}
& \leq \beta 2^{p-1} \Bigl\{ \int_Q\int_{Q_1} |V_j(x,\lambda_j\xi)|^p\,dx\,d\xi+ |M|^p|T|^p|Q|\Bigr\}.
\end{align}
Now we note that, by \eqref{f : period},
\begin{align*}
    \int_Q\int_{Q_1} |V_j(x,\lambda_j\xi)|^p\,dx\,d\xi & = \lambda_j^{-d}\int_{\lambda_j Q}\int_{Q_1}|V_j(x,\xi)|^p\,dx\,d\xi \\
    & \leq \lambda_j^{-d} \sum_{i\in \mathcal{I}_j} \int_{Q_1}\int_{Q_1}|V_j(x,\xi)|^p\,dx\,d\xi,
\end{align*}
where $\mathcal{I}_j:=\{i\in\mathbb Z^d : i+Q_1\cap \lambda_jQ\neq \emptyset\}$. Since $\lambda_j^{-d}\#\mathcal{I}_j\to |Q|$, we deduce
\begin{equation*}
    \limsup_{j\to+\infty}\int_Q\int_{Q_1} |V_j(x,\lambda_j\xi)|^p\,dx\,d\xi \leq  |Q|\limsup_{j\to+\infty}\|V_j(x,\xi)\|^p_{L^p(Q_1\times Q_1;\R^m)},
\end{equation*}
which, together with \eqref{equi1}, concludes the proof in light of Lemma \ref{lemma : supercritical 1}.

As for the sequence $\{\widetilde{g}_j\}_j$, by \eqref{f : growth} we simply have
\begin{align*}
\int_A \widetilde{g}_j(\xi)\,d\xi & = \int_A\int_{Q_1} \int_{Q_1} f(x,x+y,V_j(x,y)+M\xi)\, dx\, dy\,d\xi \\
& \leq \beta 2^{p-1} \Bigl\{\|V_j(x,y)\|^p_{L^p(Q_1\times Q_1;\R^m)}|A|+ |M|^p|T|^p|A|\Bigr\}
\end{align*}
for any $A\subset B_T$ measurable and for every $j\in\N$, and the conclusion follows as before.
\end{proof}

We state our main lemma.

\begin{lemma}\label{lemma : supercritical 3}
    It holds that
    \begin{equation*}
        g_j(\xi)- \widetilde{g}_j(\xi) \rightharpoonup 0
    \end{equation*}
    weakly in $L^1(B_T)$ as $j\to+\infty$.
\end{lemma}
\begin{proof}
    Due to the equi-integrability of the sequence in $L^1(B_T)$ established in Lemma \ref{lemma : supercritical 2}, it suffices to test the weak convergence with the characteristic functions of cubes contained in $B_T$. Let $Q$ denote such a cube, we have
\begin{align*}
\int_Q g_j(\xi)-\widetilde{g}_j(\xi)\, d\xi  & = \lambda_j^{-d} \int_{\lambda_jQ} g_j(\xi/\lambda_j)- \widetilde{g}_j(\xi/\lambda_j)\,d\xi \\
& = \lambda_j^{-d} \sum_{i\in \mathcal{I}_j}\int_{i+Q_1} g_j(\xi/\lambda_j)- \widetilde{g}_j(\xi/\lambda_j)\,d\xi \\
& \quad + \lambda_j^{-d} \int_{E_j} g_j(\xi/\lambda_j)- \widetilde{g}_j(\xi/\lambda_j)\,d\xi,
\end{align*}
where we set $\mathcal{I}_j:=\{i\in \mathbb Z^d : i+Q_1\subset \lambda_jQ\}$ and $E_j$ is obtained removing from $\lambda_jQ$ all the unit cubes of the fundamental lattice. Since there exists $c>0$ such that $|E_j|\leq c\lambda_j^{d-1}$ for all $j\in\N$, we have that
\begin{equation*}
    \lambda_j^{-d} \int_{E_j} g_j(\xi/\lambda_j)- \widetilde{g}_j(\xi/\lambda_j)\,d\xi =  \int_{\lambda_j^{-1}E_j} g_j(\xi)- \widetilde{g}_j(\xi)\,d\xi
\end{equation*}
vanishes as $j\to+\infty$ by Lemma \ref{lemma : supercritical 2}; therefore, applying a change of variables, the periodicity of each $V_j$, and \eqref{f : period}, the claim follows if we prove that
\begin{multline}\label{superclaim}
    \lim_{j\to+\infty } \lambda_j^{-d} \sum_{i\in \mathcal{I}_j}\int_{Q_1} \Bigl\{ \int_{Q_1} f\Bigl(x,x+\xi,V_j(x,\xi)+ M\frac{i+\xi}{\lambda_j}\Bigr)\, dx\\ 
     - \int_{Q_1} \int_{Q_1} f\Bigl(x,x+y,V_j(x,y)+M\frac{i+\xi}{\lambda_j}\Bigr)\, dx\, dy\Bigr\}\, d\xi =0.
\end{multline}
To this end, we first prove that
\begin{multline}\label{subsuperclaim1}
    \lim_{j\to+\infty } \lambda_j^{-d} \sum_{i\in \mathcal{I}_j}\int_{Q_1}  \int_{Q_1} f\Bigl(x,x+\xi,V_j(x,\xi)+ M\frac{i+\xi}{\lambda_j}\Bigr)\, dx\,d\xi \\ -\int_Q \int_{Q_1} \int_{Q_1} f(x,x+\xi,V_j(x,\xi)+Mz)\,dx \, d\xi\, dz =0.
\end{multline}
Indeed, we have
\begin{align}\notag
    & \hspace{-1cm}\lim_{j\to+\infty} \Bigl| \lambda_j^{-d} \sum_{i\in \mathcal{I}_j}\int_{Q_1} \int_{Q_1} f\Bigl(x,x+\xi,V_j(x,\xi)+ M\frac{i+\xi}{\lambda_j}\Bigr)\, dx\,d\xi \\ \notag
    & \hspace{5cm} -\int_Q\int_{Q_1} \int_{Q_1} f(x,x+\xi,V_j(x,\xi)+Mz)\,dx \, d\xi\, dz \Bigr|\\ \notag
    & \hspace{-1cm} = \lim_{j\to+\infty} \Bigl|\sum_{i\in \mathcal{I}_j} \int_{\frac{i+Q_1}{\lambda_j}} \Bigl\{\int_{Q_1}  \int_{Q_1} f\Bigl(x,x+\xi,V_j(x,\xi)+ M\frac{i+\xi}{\lambda_j}\Bigr)\, dx\,d\xi \\ \notag 
    & \hspace{5cm} - \int_{Q_1} \int_{Q_1} f(x,x+\xi,V_j(x,\xi)+Mz)\,dx\, d\xi \Bigr\}\, dz \Bigr| \\  \notag
    & \hspace{-1cm} \leq \lim_{j\to+\infty} \sum_{i\in \mathcal{I}_j} \int_{\frac{i+Q_1}{\lambda_j}}\int_{Q_1}\int_{Q_1}\Bigl|f\Bigl(x,x+\xi,V_j(x,\xi) + M\frac{i+\xi}{\lambda_j}\Bigr) \\ \label{superclaim-1}
    & \hspace{5cm}-f(x,x+\xi,V_j(x,\xi)+Mz)\Bigr|\,dx\,d\xi\,dz,
\end{align}
where we used that, by Lemma \ref{lemma : supercritical 2},
\begin{align*}
     \lim_{j\to+\infty} \int_{Q\setminus \bigcup_{i\in\mathcal{I}_j} \frac{i+Q_1}{\lambda_j} }\Bigr\{\int_{Q_1} \int_{Q_1} f(x,x+\xi,V_j(x,\xi)+Mz)\,dx\, d\xi \Bigr\}\, dz = \lim_{j\to+\infty} \int_{Q\setminus \bigcup_{i\in\mathcal{I}_j} \frac{i+Q_1}{\lambda_j} } \widetilde{g}_j(z)\,dz=0.
\end{align*}

We recall that, by \eqref{f : convex}, the function $f$ is locally Lipschitz continuous in the last variable and, in particular, resorting also to \eqref{f : growth}, it holds that
\begin{equation}\label{loclip}
    |f(x,y,z_1)-f(z,y,z_2)|\leq C(1+|z_1|^{p-1}+|z_2|^{p-1})|z_1-z_2|
\end{equation}
for a.e.\! $x,y\in\R^d$ and for every $z_1,z_2\in \R^m$, where $C$ is a constant that depends on $p,d,$ and the constant $\beta$ appearing in \eqref{f : growth}. Let then $i\in \mathcal{I}_j$ and $z\in \frac{i+Q_1}{\lambda_j}$; using \eqref{loclip} we have
\begin{align*}
    &   \Bigl|f\Bigl(x,x+\xi,V_j(x,\xi) + M\frac{i+\xi}{\lambda_j}\Bigr)-f(x,x+\xi,V_j(x,\xi)+Mz)\Bigr| \\
    & \leq C\Bigl(1+\Bigl|V_j(x,\xi)+ M\frac{i+\xi}{\lambda_j}\Bigr|^{p-1}+|V_j(x,\xi)+Mz|^{p-1}\Bigr)\Bigl|M\Bigl(\frac{i+\xi}{\lambda_j}-z\Bigr)\Bigr| \\
    &\leq C'(1+|V_j(x,\xi)|^{p-1})\lambda_j^{-1}
\end{align*}
for a.e.\! $x,\xi\in Q_1$, where $C'$ is a positive constant depending on $p,d,\beta,$ and $M$. Combining this estimate with \eqref{superclaim-1} and using H\"older's inequality, we infer
\begin{align*}
    & \hspace{-1cm} \lim_{j\to+\infty } \Bigl|\lambda_j^{-d} \sum_{i\in \mathcal{I}_j}\int_{Q_1} \int_{Q_1} f\Bigl(x,x+\xi,V_j(x,\xi) + M\frac{i+\xi}{\lambda_j}\Bigr)\, dx\,d\xi \\ \notag
    & \hspace{4cm} -\int_Q \int_{Q_1} \int_{Q_1} f(x,x+\xi,V_j(x,\xi)+Mz)\, d\xi\,dx \, dz \Bigr| \\ \notag
    & \hspace{-1cm} \leq \lim_{j\to+\infty}C'\lambda_j^{-1}\int_Q\int_{Q_1}\int_{Q_1} 1+|V_j(x,\xi)|^{p-1}\, dx\,d\xi\,dz \\
    & \hspace{-1cm} \leq C'|Q|\lim_{j\to+\infty}\lambda_j^{-1}(1+\|V_j(x,\xi)\|^p_{L^p(Q_1\times Q_1;\R^m)}),
\end{align*}
which equals $0$ by Lemma \ref{lemma : supercritical 1}. Then, \eqref{subsuperclaim1} is proved.

A similar argument shows that
\begin{multline} \label{subsuperclaim2}
     \lim_{j\to+\infty } \lambda_j^{-d} \sum_{i\in \mathcal{I}_j}\int_{Q_1}   \int_{Q_1} \int_{Q_1} f\Bigl(x,x+y,V_j(x,y)+M\frac{i+\xi}{\lambda_j}\Bigr)\, dx \, dy\,d\xi \\ 
     -\int_Q \int_{Q_1} \int_{Q_1} f(x,x+y,V_j(x,y)+Mz)\,dx \, dy\, dz =0.
\end{multline}
Indeed, arguing as before, we have
\begin{align*}
    & \hspace{-1cm} \lim_{j\to+\infty } \Bigl| \lambda_j^{-d} \sum_{i\in \mathcal{I}_j}\int_{Q_1} \int_{Q_1} \int_{Q_1} f\Bigl(x,x+y,V_j(x,y)+M\frac{i+\xi}{\lambda_j}\, \Bigr)dx \, dy\, d\xi \\
    & \hspace{5cm} -\int_Q \int_{Q_1} \int_{Q_1} f(x,x+y,V_j(x,y)+Mz)\,dx \, dy\, dz \Bigr| \\
&  \hspace{-1cm} = \lim_{j\to+\infty} \Bigl| \sum_{i\in \mathcal{I}_j} \int_{\frac{i+Q_1}{\lambda_j}} \Bigl\{ \int_{Q_1} \int_{Q_1} \int_{Q_1} f\Bigl(x,x+y,V_j(x,y)+M\frac{i+\xi}{\lambda_j}\Bigr)\, dx \, dy\, d\xi \\
& \hspace{5cm} -\int_{Q_1}\int_{Q_1} f(x,x+y,V_j(x,y)+Mz)\,dx \, dy \Bigr\}\,dz \Bigr| \\
    & \hspace{-1cm} \leq \lim_{j\to+\infty}  \sum_{i\in \mathcal{I}_j} \int_{\frac{i+Q_1}{\lambda_j}}\Bigl\{ \int_{Q_1}\Bigl[\int_{Q_1}\int_{Q_1}\Bigl|f\Bigl(x,x+y,V_j(x,y)+M\frac{i+\xi}{\lambda_j}\Bigr) \\
    & \hspace{5cm} - f(x,x+y,V_j(x,y)+Mz) \Bigr|\,dx\,dy\Bigr]\,d\xi \Bigr\}\,dz \\
    & \hspace{-1cm} \leq C'|Q|\lim_{j\to+\infty} \lambda_j^{-1}(1+\|V_j(x,y)\|^p_{L^p(Q_1\times Q_1;\R^m)}),
\end{align*}
which equals $0$, and this proves \eqref{subsuperclaim2}.

Gathering \eqref{subsuperclaim1} and \eqref{subsuperclaim2} we obtain \eqref{superclaim}, which is the thesis.
\end{proof}

We are now in position to prove the lower bound. Recalling \eqref{superstarting} and the definitions of $g_j$ and $\widetilde{g}_j, j\in\N,$ we fix any $R>0$ and apply Lemma \ref{lemma : supercritical 3} to obtain
\begin{align*}
    \f(M) & \geq \limsup_{j\to+\infty} \int_{B_T} \rho(\xi)g_j(\xi)\,d\xi \\
    & \geq \liminf_{j\to+\infty} \int_{B_T} \min\{\rho(\xi),R\}(g_j(\xi)-\widetilde{g}_j(\xi))\,d\xi + \int_{B_T} \min\{\rho(\xi),R\}\widetilde{g}_j(\xi)\,d\xi \\
    & = \liminf_{j\to+\infty}  \int_{B_T} \min\{\rho(\xi),R\}\widetilde{g}_j(\xi)\,d\xi \\
    & = \liminf_{j\to+\infty}  \int_{B_T} \min\{\rho(\xi),R\} \int_{Q_1} \int_{Q_1} f(x,x+y,V_j(x,y)+M\xi)\, dx\, dy \,d\xi.
    \end{align*}
    Using Lemma \ref{lemma : supercritical 1}, it is not restrictive to suppose that $V_j(x,y)\rightharpoonup V(x,y)$ weakly in $L^p(Q_1\times Q_1;\R^m)$ and, setting
\begin{align*}
    \mathcal{C}:=\Bigl\{U\in  L^p_{\text{loc}}(\R^d\times \R^d; \R^m)  :\, & U(\cdot, y) \text{ is } Q_1\text{-periodic for a.e. } y\in \R^d, \\
    & U(x,\cdot) \text{ is } Q_1\text{-periodic for a.e. } x\in \R^d,  \\
    & \text{and } \int_{Q_1}U(x,y)\,dx=0 \text{ for a.e. } y\in \R^d \Bigr\},
\end{align*}
by the definition of $V_j(x,y)$ and the periodicity of $v_j(x)$, it is easily seen that $V\in \mathcal{C}$. Then, we apply Theorem \ref{thm : LSC} with $\ell=d\times d, s=(x,y)$ and
\begin{equation*}
   E=Q_1\times Q_1,  \quad  \Psi((x,y),t,q)=f(x,x+y,q+M\xi),
\end{equation*} and use Fatou's Lemma to infer
\begin{align*}
    \f(M) & \geq \int_{B_T} \min\{\rho(\xi),R\} \liminf_{j\to+\infty}  \int_{Q_1} \int_{Q_1} f(x,x+y,V_j(x,y)+M\xi)\, dx\, dy \,d\xi \\
    & \geq \int_{B_T} \min\{\rho(\xi),R\} \int_{Q_1} \int_{Q_1} f(x,x+y,V(x,y)+M\xi)\, dx\, dy \,d\xi,
\end{align*}
which, by the arbitrariness of $R>0$ and $V\in \mathcal{C}$, implies
\begin{equation*}
    \f(M) \geq \inf\Bigl\{ \int_{B_T} \rho(\xi)\int_{Q_1}\int_{Q_1} f(x,x+y,V(x,y)+M\xi)\, dx\,dy\,d\xi : V\in\mathcal{C}\Bigr\}.
\end{equation*}

For the sake of notation, we set
\begin{equation*}
        F(V):=\int_{B_T} \rho(\xi)\int_{Q_1}\int_{Q_1} f(x,x+y,V(x,y)+M\xi)\, dx\,dy\,d\xi  
    \end{equation*}
    so that
    \begin{equation*}
    \f(M) \geq \inf\{ F(V) : V\in\mathcal{C}\}.
\end{equation*}
Let us suppose first that
\begin{equation*}
    \tag{H1}
        \rho(\xi)f(x,y,z)=\rho(-\xi)f(x,y,-z) 
    \end{equation*}
    for a.e. $\xi,x,y\in \R^d$ and for every $z\in\R^m$. Using a change of variables and \eqref{H1}, we have that
    \begin{align*}
        F(V) & =\int_{B_T} \rho(\xi)\int_{Q_1} f(x,x+y,V(x,y)+M\xi)\, dx\,dy\,d\xi \\
        & = \int_{B_T} \rho(-\xi)\int_{Q_1} f(x,x+y,V(x,y)-M\xi)\, dx\,dy\,d\xi \\
        & = \int_{B_T} \rho(\xi)\int_{Q_1} f(x,x+y, -V(x,y)+M\xi)\, dx\,dy\,d\xi \\
        & = F(-V);
    \end{align*}
    therefore, recalling \eqref{f : convex}, we infer
    \begin{equation*}
        F(0)\leq \frac{F(V)+F(-V)}{2} = F(V).
    \end{equation*}
    The arbitrariness of $V$ yields
    \begin{equation}\label{superfinal}
        \f(M)\geq F(0)=\int_{B_T} \rho(\xi)\int_{Q_1}\int_{Q_1} f(x,x+y,M\xi)\, dx\,dy\,d\xi ,
    \end{equation}
    which, using Fubini's Theorem and \eqref{f : period}, yields the desired lower bound.

\smallskip

Assume now that
\begin{equation*}
    \tag{H2}
    \int_{\R^d}\rho(\xi)\,d\xi =+\infty
\end{equation*}
and let $V\in \mathcal{C}$. The lower bound in \eqref{f : growth} implies that
\begin{equation*}
    \rho(\xi)f(x,x+y, V(x,y)+M\xi) \geq \alpha \rho(\xi)|V(x,y)+M\xi|^p
\end{equation*}
for a.e. $\xi,x,y\in \R^d$ and for every $z\in\R^m$,
and since there exist positive constants $c_1,c_2$ such that $|z_1+z_2|^p\geq c_1|z_1|^p-c_2|z_2|^p$ for every $z_1,z_2\in\R^m$, we infer that
\begin{equation*}
    \rho(\xi)f(x,x+y, V(x,y)+M\xi) + \alpha c_2 \rho(\xi)|M\xi|^p \geq \alpha c_1 \rho(\xi)|V(x,y)|^p.
\end{equation*}
Integrating this inequality, we get
\begin{equation*}
 F(V) +\alpha c_2\int_{B_T} \rho(\xi)|M\xi|^p\,d\xi  \geq \alpha c_1\Bigl(\int_{B_T} \rho(\xi)\,d\xi\Bigr)\Bigl(\int_{Q_1}\int_{Q_1} |V(x,y)|^p\, dx\,dy\Bigr);
\end{equation*}
therefore, by \eqref{rho : moment} and \eqref{H2}, we get that $F(V)$ is finite only if $V=0$, which implies \eqref{superfinal}.

\subsubsection{Upper bound}

In order to prove the optimality of the previous lower bound, we recall that, according to $(i)$ of Proposition \ref{prop : cell formula}, we have
\begin{align*}
    \f(M) & = \lim_{j\to+\infty} \inf\Bigl\{ \Bigl(\frac{\dej}{r}\Bigr)^d\int_{B_T} \rho(\xi)\int_{(Q_{r/\dej})_{\lambda_j}(\xi)} f\Bigl(x, x+\lambda_j\xi, \frac{u(x+\lambda_j\xi)-u(x)}{\lambda_j}\Bigr)\, dx\,d\xi : \\
    & \qquad \qquad \qquad u\in \mathcal{D}_{T\lambda_j,M}(Q_{r/\dej};\R^m) \Bigr\}
\end{align*}
for some $r>0$. We use $u(x)=Mx$ as a test function for these minimum problems to obtain
\begin{align}\nonumber
    \f(M) & \leq \liminf_{j\to+\infty} \Bigl(\frac{\dej}{r}\Bigr)^d\int_{B_T} \rho(\xi)\int_{(Q_{r/\dej})_{\lambda_j}(\xi)} f(x, x+\lambda_j\xi, M\xi)\, dx\,d\xi\\ \label{supsuper}
    & \leq \liminf_{j\to+\infty} \int_{B_T} \rho(\xi)\int_{Q_1} f(x,x+\lambda_j\xi,M\xi)\, dx\,d\xi ,
\end{align}
where the last inequality follows by \eqref{f : period}. 

On the one hand, if we apply Lemma \ref{lemma : supercritical 3} with $V_j=0$ for every $j\in \N$, we have that
\begin{equation*}
      g_j(\xi)=\int_{Q_1} f(x,x+\lambda_j\xi,M\xi)\, dx
\rightharpoonup
    \int_{Q_1} \int_{Q_1} f(x,y,M\xi)\, dx\, dy
\end{equation*}
weakly in $L^1(B_T)$. 

On the other hand, we observe that $\{\rho g_j\}_j$ is an equi-integrable sequence in $L^1(B_T)$ since, by \eqref{f : growth}, we have 
\begin{equation*}
    \int_A \rho(\xi)g_j(\xi)\,d\xi \leq \beta|M|^p \int_A\rho(\xi)|\xi|^p\,d\xi
\end{equation*}
for every $A\subset B_T$ measurable and $\rho$ has finite $p$-th moment according to \eqref{rho : moment}. Then, it is not restrictive to suppose that $\rho g_j\rightharpoonup \Theta$ weakly in $L^1(B_T)$ as $j\to+\infty$.

We claim that
\begin{equation}\label{supersupclaim}
   \rho(\xi) \int_{Q_1} \int_{Q_1} f(x,y,M\xi)\, dx\, dy=  \Theta(\xi)
   \end{equation}
for a.e. $\xi\in B_T$. Let $\chi_R:=\chi_{\{\xi\in B_T: \rho(\xi)<R\}}$ and $\psi\in C^\infty_c(B_T)$. Since $\rho\chi_R\psi\in L^\infty(B_T)$, by the weak convergence of $\{g_j\}_j$ we infer
\begin{equation*}
    \lim_{j\to+\infty} \int_{B_T}\rho(\xi)\chi_R(\xi)\psi(\xi)g_j(\xi)\,d\xi = \int_{B_T}\rho(\xi)\chi_R(\xi)\psi(\xi)\Bigl(\int_{Q_1} \int_{Q_1} f(x,y,M\xi)\, dx\, dy\Bigr)\,d\xi,
\end{equation*}
and analogously, since $\chi_R\psi\in L^\infty(B_T)$, we have
\begin{equation*}
    \lim_{j\to+\infty} \int_{B_T}\rho(\xi)\chi_R(\xi)\psi(\xi)g_j(\xi)\,d\xi = \int_{B_T}\Theta(\xi)\chi_R(\xi)\psi(\xi)\,d\xi.
\end{equation*}
We deduce that 
\begin{equation*}
   \rho(\xi)\int_{Q_1} \int_{Q_1} f(x,y,M\xi)\, dx\, dy= \Theta(\xi)
\end{equation*}
for a.e. $\xi\in B_T$ such that $\rho(\xi)<R$ and for every $R>0$; and since the set $\{\xi\in \R^d:\rho(\xi)=+\infty\}$ has measure zero in virtue of \eqref{rho : moment}, we infer \eqref{supersupclaim}.

Recalling \eqref{supsuper} we conclude
\begin{align*}
    \f(M) & \leq \liminf_{j\to+\infty}  \int_{B_T} \rho(\xi)\int_{Q_1} f(x,x+\lambda_j\xi,M\xi)\, dx\,d\xi \\
    & = \lim_{j\to+\infty}  \int_{B_T} \rho(\xi)g_j(\xi)d\xi \\
    & = \int_{B_T}\rho(\xi)\int_{Q_1} \int_{Q_1} f(x,y,M\xi)\, dx\, dy\,d\xi,
\end{align*}
which proves the upper bound.

\begin{remark}\label{rmk super} 
As a byproduct, the above arguments show that
\begin{multline*}
    \int_{B_T}\int_{Q_1} \int_{Q_1} \rho(\xi) f(x,y,M\xi)\, dx\, dy\,d\xi \\ = \lim_{j\to+\infty} \inf\Bigl\{ \int_{B_T} \rho(\xi)\int_{Q_1} f\Bigl(x,x+\lambda_j\xi,\frac{u(x+\lambda_j\xi)-u(x)}{\lambda_j}\Bigr)\, dx\,d\xi : u\in L^p_{\#,M}(Q_1;\R^m) \Bigr\},
\end{multline*}
and therefore, although in an indirect way, we obtain that $(iii)$ of Proposition \ref{prop : cell formula} is valid also for $\lambda=+\infty$. We point out that, in this case, we also had to also resort to the assumption \eqref{H1} or \eqref{H2}.
\end{remark}

\subsubsection{Further comments on the supercritical case}

We discuss our necessity to assume \eqref{H1} or \eqref{H2} in order to treat the supercritical case $\lambda=+\infty$. 
We just proved that if a certain $\f$ is a density for the $\Gamma$-limit, then
\begin{equation*}
    \f(M) \geq \inf\Bigl\{ \int_{B_T} \rho(\xi)\int_{Q_1}\int_{Q_1} f(x,x+y,V(x,y)+M\xi)\, dx\,dy\,d\xi : V\in\mathcal{C}\Bigr\}
\end{equation*}
for all $M\in \R^{m\times d}$. Inspecting the proof, we note that our argument is only based on the equi-boundedness of the difference quotients
\begin{equation*}
   V_j(x,\xi):= \frac{v_j(x+\xi)-v_j(x)}{\lambda_j}, \quad j\in\N,
\end{equation*}
which implies that (up to subsequences)
\begin{equation*}
    \frac{v_j(x+\xi)-v_j(x)}{\lambda_j} \rightharpoonup V(x,\xi)
\end{equation*}
weakly in $L^p(Q_1\times Q_1;\R^m)$ for some $V\in\mathcal{C}$,
but it does not exploit the fact that $\{u_j(x)=v_j(x)+Mx\}_j$ is a sequence of minimizers for the problems appearing in $(ii)$ of Proposition \ref{prop : cell formula}. This may lead us to believe that there is some room for an improvement: for instance, we may suspect that, incorporating the optimality of $\{u_j\}_j$, it may be possible to directly prove that 
\begin{equation}\label{zerofinale}
    \frac{v_j(x+\xi)-v_j(x)}{\lambda_j} \rightharpoonup 0,
\end{equation}
which would allow us to prove {\em (iii)} of Theorem \ref{thm : main} without any additional assumptions. According to the next Proposition, this may not be the case, even for a simple choice of the density $f$. 

Let $\{\e_j\}_j, \{\dej\}_j$ be vanishing sequences such that $\lambda_j:=\ej/\dej\to+\infty$ as $j\to+\infty$ and consider
\begin{equation*}
      F_j(u)=\int_\Omega\int_\Omega \frac{1}{\ej^d}\rho\Bigl(\frac{y-x}{\ej}\Bigr)a\Bigl(\frac{x}{\dej}\Bigr)\Bigl|\frac{u(y)-u(x)}{\ej}\Bigr|^2\, dx\,dy,    
\end{equation*}
with $\rho$ a kernel in $L^1(\R^d)\cap L^\infty(\R^d)$ supported on $B_T$ for some $T>0$ and fulfilling \eqref{rho : ball} and \eqref{rho : moment}, and $a$ a non-constant $Q_1$-periodic function such that $0<\alpha\leq a(x)\leq\beta<+\infty$ for a.e. $x\in \R^d$. With this choice, we have that $f(x,y,z)=a(x)|z|^2$, which implies that $f(x,y,z)=f(x,y,-z)$ for a.e. $x,y\in \R^d$ and for every $z\in\R$, where we work in the scalar case $m=1$ for the sake of simplicity. Therefore, we have that \eqref{H1} is equivalent to
\begin{equation}\label{Ht1} \tag{$\rm \widetilde{H}1$}
    \rho(\xi)=\rho(-\xi)
\end{equation}
for a.e. $\xi\in B_T$ and that \eqref{H2} fails.

\begin{proposition}\label{prop : superobstruction}
Let $M\in \R^{d}, M\neq 0$, and, for all $j\in\N$, let $u_j$ be a minimizer for    
\begin{equation}\label{superminima}
    \inf\Bigl\{ \int_{B_T} \rho(\xi)\int_{Q_1} a(x)\Bigl|\frac{u(x+\lambda_j\xi)-u(x)}{\lambda_j}\Bigr|^2\, dx\,d\xi : u\in L^2_{\#,M}(Q_1) \Bigr\},
\end{equation}
let $v_j(x):=u_j(x)-Mx,$ and let 
\begin{equation*}
    V_j(x,\xi):=\frac{v_j(x+\xi)-v_j(x)}{\lambda_j}.
\end{equation*}
If $V_j\rightharpoonup 0$ weakly in $L^2(Q_1\times Q_1)$ as $j\to+\infty$, then 
\begin{equation*}
    \int_{B_T}\rho(\xi)M\xi\,d\xi=0.
\end{equation*}
\end{proposition}
\begin{proof}
Since $\rho\in L^1(\R^d)$, it is possible to compute the Euler-Lagrange equations for the minimum problems \eqref{superminima} obtaining that for every $j\in\N$ we have
\begin{equation*}
    \int_{B_T} \rho(\xi)\int_{Q_1} a(x)\Bigl(\frac{u_j(x+\lambda_j\xi)-u_j(x)}{\lambda_j}\Bigr)\Bigl(\frac{w(x+\lambda_j\xi)-w(x)}{\lambda_j}\Bigr)\, dx\,d\xi=0 
\end{equation*}
for every $w\in L^2_{\#,0}(Q_1)$.

By first testing this equation with $w(x)=v_j(x)$, we infer
\begin{multline}
    \int_{B_T} \rho(\xi)\int_{Q_1} a(x)\Bigl|\frac{v_j(x+\lambda_j\xi)-v_j(x)}{\lambda_j}\Bigr|^2\, dx\,d\xi \\ \label{ele1}
    = - \int_{B_T} \rho(\xi)M\xi\int_{Q_1} a(x)\Bigl(\frac{v_j(x+\lambda_j\xi)-v_j(x)}{\lambda_j}\Bigr)\, dx\,d\xi;
\end{multline}
then, testing with $w(x)=\lambda_j\widetilde{w}(x)$ for some $\widetilde{w}\in L^2_{\#,0}(Q_1)$, we obtain
\begin{multline}\label{ele2}
     \int_{B_T} \rho(\xi)\int_{Q_1} a(x)\Bigl(\frac{v_j(x+\lambda_j\xi)-v_j(x)}{\lambda_j}\Bigr)(\widetilde{w}(x+\lambda_j\xi)-\widetilde{w}(x))\, dx\,d\xi \\
     = - \int_{B_T} \rho(\xi)M\xi\int_{Q_1} a(x)(\widetilde{w}(x+\lambda_j\xi)-\widetilde{w}(x))\, dx\,d\xi.
\end{multline}

Now we note that $V_j(x,\lambda_j\xi)\rightharpoonup 0$ weakly in $L^2(Q_1\times B_T)$. Indeed, by \eqref{superbound}, this sequence is equi-bounded in $L^2(Q_1\times B_T)$, hence, it suffices to test the weak convergence with the characteristic functions of measurable sets $E\times Q\subset Q_1\times B_T$, with $Q$ a cube. Also using the equi-integrability of $\{V_j\}_j$ in $L^1(Q_1\times Q_1)$, it is immediate to observe that
\begin{align*}
    \lim_{j\to+\infty} \int_Q\int_E V_j(x,\lambda_j\xi)\,dx\,d\xi & =  \lim_{j\to+\infty} \lambda_j^{-d}\int_{\lambda_jQ}\int_E V_j(x,\xi)\,dx\,d\xi \\
    & = |Q|\lim_{j\to+\infty} \int_{Q_1}\int_E V_j(x,\xi)\,dx\,d\xi,
\end{align*}
which equals $0$ since $V_j\rightharpoonup 0$ by assumption. Passing to the limit in \eqref{ele1}, we get 
\begin{equation*}
   \lim_{j\to+\infty} \int_{B_T} \rho(\xi)\int_{Q_1} a(x)\Bigl|\frac{v_j(x+\lambda_j\xi)-v_j(x)}{\lambda_j}\Bigr|^2\, dx\,d\xi = 0,
\end{equation*}
which, by \eqref{rho : ball} and the fact that $\inf a\geq \alpha$, implies
\begin{equation*}
   \lim_{j\to+\infty} \int_{B_{r_0}} \int_{Q_1} \Bigl|\frac{v_j(x+\lambda_j\xi)-v_j(x)}{\lambda_j}\Bigr|^2\, dx\,d\xi = 0,
\end{equation*}
and then, by Lemma \ref{lemma : diadic},
\begin{equation*}
   \lim_{j\to+\infty} \int_{B_T} \int_{Q_1} \Bigl|\frac{v_j(x+\lambda_j\xi)-v_j(x)}{\lambda_j}\Bigr|^2\, dx\,d\xi = 0.
\end{equation*}

Upon assuming that $\int_{Q_1}\widetilde{w}=0$, by the Riemann-Lebesgue Lemma we have 
\begin{equation*}
    \widetilde{w}(x+\lambda_j\xi)-\widetilde{w}(x)\rightharpoonup -\widetilde{w}(x)
\end{equation*}
weakly in $L^2(Q_1\times B_T)$ as $j\to+\infty$; hence, we pass to the limit in \eqref{ele2} using the strong-weak convergence to obtain
\begin{equation*}
    0 =\int_{B_T} \rho(\xi)M\xi\,d\xi\int_{Q_1} a(x)\widetilde{w}(x)\, dx. 
\end{equation*}
Finally, properly choosing $\widetilde{w}$, we get
\begin{equation*}
    \int_{B_T} \rho(\xi)M\xi\,d\xi=0,
\end{equation*}
which is the thesis.
\end{proof}

In this setting, if we choose $\rho$ in such a way that also \eqref{Ht1} fails, we find a certain $\widetilde{M}\neq 0$ for which
\begin{equation*}
    \int_{B_T} \rho(\xi)\widetilde{M}\xi\,d\xi\neq0;
\end{equation*}
and then, applying the above Proposition, we deduce that 
\eqref{zerofinale} fails for the sequence $\{v_j\}_j$ corresponding to such $\widetilde{M}$, as claimed.

\section{Proof of the main result}

Finally, we prove Theorem \ref{thm : main} removing the truncation assumption on the functionals.

\begin{proof}[Proof of Theorem \ref{thm : main}]

Let $\{T_h\}_h$ be any sequence monotonically increasing to $+\infty$ and assume that $T_h>r_0$ and $Q_1\subset B_{T_h}$ for all $h\in\N$, where $r_0$ is as in \eqref{rho : ball}. By Theorem \ref{thm : intrep} and a diagonal argument, there exists a subsequence $\{j_k\}_k$ such that, for every $h\in\N$, the sequence of truncated functionals $\{F_{j_k}^{T_h}(\cdot,\cdot)\}_k$ $\Gamma$-converges to a functional that admits the integral representation 
\begin{equation}\label{representation}
   F^{T_h}(u,A) := \Gamma(L^p)\text{-}\lim_{k\to+\infty} F_{j_k}^{T_h}(u,A) = \begin{cases}
            \displaystyle \int_{A} \f_h(\nabla u)\,dx & \text{ if } u\in W^{1,p}(A;\R^m), \\
            +\infty & \text{ if } u\in L^p(\Omega;\R^m)\setminus W^{1,p}(A;\R^m)
        \end{cases}
\end{equation}
for all $A\in \A_{\rm reg}(\Omega)$ and for some quasiconvex $\f_h: \R^{m\times d}\to [0,+\infty)$. By Lemma \ref{lemma : truncated functionals}, we have that \begin{equation*}
   \Gamma(L^p)\text{-}\lim_{k\to+\infty} F_{j_k}(u,A) =\lim_{h\to+\infty} F^{T_h}(u,A) 
\end{equation*}
for every $u\in L^p(\Omega; \R^m)$ and $A\in \A_{\rm reg}(\Omega)$. Clearly, $\{\f_h\}_h$ is an increasing sequence of functions; therefore, by \eqref{representation} and the Monotone Convergence Theorem, we have
\begin{equation*}
   \Gamma(L^p)\text{-}\lim_{k\to+\infty} F_{j_k}(u,A) =  \begin{cases}
            \displaystyle \int_{A} \lim_{h\to+\infty}\f_h(\nabla u)\,dx & \text{ if } u\in W^{1,p}(A;\R^m), \\
            +\infty & \text{ if } u\in L^p(\Omega;\R^m)\setminus W^{1,p}(A;\R^m)
            \end{cases}
\end{equation*}
for all $A\in \A_{\rm reg}(\Omega)$. In order to conclude, we prove that
\begin{equation*}
    \lim_{h\to+\infty}\f_h(M)= f_\lambda(M)
\end{equation*}
for all $M\in \R^{m \times d}$, where $f_\lambda$ is defined as in the statement of Theorem \ref{thm : main} in accordance with the value of $\lambda\in [0,+\infty]$. Indeed, since $f_\lambda$ is independent of the subsequence $\{j_k\}_k$, the conclusion follows by the Urysohn property of the $\Gamma$-convergence (see \cite{DM}).

\medskip

{\em Subcritical case, $\lambda=0$.} Let $M\in \R^{m\times d}$ be fixed. By Proposition \ref{prop : main}, we have that
\begin{align*}
    \lim _{h\to+\infty}\f_h(M) & = \lim_{h\to+\infty}\inf\Bigl\{ \int_{B_{T_h}} \rho(\xi)\int_{Q_1} f(x,x,(\nabla u ) \xi)\, dx\,d\xi : u\in W^{1,p}_{\#,M}(Q_1; \R^m) \Bigr\} \\
    & \leq \inf\Bigl\{ \int_{\R^d} \rho(\xi)\int_{Q_1} f(x,x,(\nabla u ) \xi)\, dx\,d\xi : u\in W^{1,p}_{\#,M}(Q_1; \R^m) \Bigr\} \\
    & = f_0(M).
\end{align*}
Conversely, let $\{u_h\}_h\subset W^{1,p}_{\#,M}(Q_1; \R^m)$ be such that
\begin{equation}\label{minimizer}
   \lim _{h\to+\infty} \f_h(M)  = \lim_{h\to+\infty}  \int_{B_{T_h}} \rho(\xi)\int_{Q_1} f(x,x,(\nabla u_h ) \xi)\, dx\,d\xi,
\end{equation}
and, without loss of generality, assume
\begin{equation}\label{null value}
\int_{Q_1}u_h\,dx=0   
\end{equation}
for every $h\in\N$. Combining \eqref{minimizer} with \eqref{rho : ball} and \eqref{f : growth}, we have
\begin{equation*}
     \sup_h \int_{Q_1} \int_{B_{r_0}}|(\nabla u_h ) \xi|^p\,d\xi\, dx<+\infty;
\end{equation*}
and since there exists a positive constant $C$ such that 
\begin{equation*}
    \int_{B_{r_0}}|L\xi|^p\,d\xi \geq C|L|^p
\end{equation*}
for every $L\in \R^{m\times d}$, we infer
\begin{equation*}
     \sup_h \int_{Q_1} |\nabla u_h|^p\, dx< +\infty.
\end{equation*}
This, together with \eqref{null value}, implies that there exists a subsequence $\{u_{h_i}\}_i$ converging to a certain $u\in W^{1,p}_{\#,M}(Q_1; \R^m)$ weakly in $W^{1,p}(Q_1;\R^m)$. As a consequence, we get that, for any $R>0$ fixed, 
\begin{equation*}
    (\nabla u_{h_i}(x))\xi \rightharpoonup (\nabla u(x))\xi 
\end{equation*}
weakly in $L^p(Q_1\times B_R;\R^m)$ as $i\to+\infty$.

Since $f$ satisfies \eqref{f : convex} and \eqref{f : continuous}, we apply Theorem \ref{thm : LSC} with $\ell=d\times d, s=(\xi,x)$, and
\begin{equation*}
    E=Q_1\times B_R, \quad \Psi((\xi,x),y,z)=\rho(\xi)f(x,y,z), 
\end{equation*}
to obtain 
\begin{align*}
     \lim _{h\to+\infty}\f_h(M) & =  \lim _{i\to+\infty} \int_{B_{T_{h_i}}} \rho(\xi)\int_{Q_1} f(x,x,(\nabla u_{h_i} ) \xi)\, dx\,d\xi \\
     & \geq \int_{B_R} \rho(\xi)\int_{Q_1} f(x,x,(\nabla u ) \xi)\, dx\,d\xi, \end{align*}
     and, by the arbitrariness of $R>0$,
\begin{align*}
  \lim _{h\to+\infty}\f_h(M) & \geq \int_{\R^d} \rho(\xi)\int_{Q_1} f(x,x,(\nabla u ) \xi)\, dx\,d\xi \\
  & \geq \inf\Bigl\{ \int_{\R^d} \rho(\xi)\int_{Q_1} f(x,x,(\nabla u ) \xi)\, dx\,d\xi : u\in W^{1,p}_{\#,M}(Q_1; \R^m) \Bigr\} \\
    & = f_0(M),
\end{align*}
concluding the proof in the subcritical case.

\medskip

{\em Critical case, $\lambda\in(0,+\infty)$.} Fix $M\in \R^{m\times d}$. By Proposition \ref{prop : main}, we have that
    \begin{align*}
      \lim_{h\to+\infty} \f_h(M) & =  \lim_{h\to+\infty} \inf\Bigl\{ \int_{B_{T_h}} \rho(\xi)\int_{Q_1} f\Bigl(x,x+\lambda\xi,\frac{u(x+\lambda\xi)-u(x)}{\lambda}\Bigr)\, dx\, d\xi : u\in L^p_{\#,M}(Q_1;\R^m) \Bigr\} \\
        & \leq \inf\Bigl\{ \int_{\R^d} \rho(\xi)\int_{Q_1} f\Bigl(x,x+\lambda\xi,\frac{u(x+\lambda\xi)-u(x)}{\lambda}\Bigr)\, dx\, d\xi : u\in L^p_{\#,M}(Q_1;\R^m) \Bigr\}.
    \end{align*}
To prove the converse, let $\{u_h\}_h\subset L^p_{\#,M}(Q_1;\R^m)$ be such that
\begin{equation*}
    \lim_{h\to+\infty} \f_h(M) =\lim_{h\to+\infty} \int_{B_{T_h}} \rho(\xi)\int_{Q_1} f\Bigl(x,x+\lambda\xi,\frac{u_h(x+\lambda\xi)-u_h(x)}{\lambda}\Bigr)\, dx\, d\xi 
\end{equation*}
and suppose that \eqref{null value} holds for every $h\in \N$. Reasoning as in the subcritical case, we have
\begin{equation*}
     \sup_h \int_{B_{r_0}} \int_{Q_1} \Bigl|\frac{u_h(x+\lambda\xi)-u_h(x)}{\lambda}\Bigr|^p\, dx\,d\xi  <+\infty,
\end{equation*}
and then, applying $(ii)$ of Lemma \ref{lemma : compactness} with $A=Q_1$ and $r=r_0$, we obtain 
that there exists a subsequence $\{u_{h_i}\}_i$ and a function $u\in L^p_{\#,M}(Q_1;\R^m)$ such that
\begin{equation*}
     \frac{u_{h_i}(x+\lambda\xi)-u_{h_i}(x)}{\lambda}\rightharpoonup \frac{u(x+\lambda\xi)-u(x)}{\lambda} 
\end{equation*}
weakly in $L^p(Q_1\times B_R;\R^m)$ as $i\to+\infty$, for any $R>0$ fixed. The conclusion now follows as for the subcritical case. Indeed, we apply Theorem \ref{thm : LSC} in the same way to obtain 
\begin{align*}
       \lim_{h\to+\infty} \f_h(M) & = \lim_{i\to+\infty} \int_{B_{T_{h_i}}} \rho(\xi)\int_{Q_1} f\Bigl(x,x+\lambda\xi,\frac{u_{h_i}(x+\lambda\xi)-u_{h_i}(x)}{\lambda}\Bigr)\, dx\, d\xi \\
      & \geq \int_{B_R} \rho(\xi)\int_{Q_1} f\Bigl(x,x+\lambda\xi,\frac{u(x+\lambda\xi)-u(x)}{\lambda}\Bigr)\, dx\, d\xi, 
      \end{align*}
and, by the arbitrariness of $R$ and $u$, we conclude that
\begin{align*}
     \lim_{h\to+\infty} \f_h(M) & \geq\inf\Bigl\{ \int_{\R^d} \rho(\xi)\int_{Q_1} f\Bigl(x,x+\lambda\xi,\frac{u(x+\lambda\xi)-u(x)}{\lambda}\Bigr)\, dx\, d\xi : u\in L^p_{\#,M}(Q_1;\R^m) \Bigr\}.
\end{align*}
Finally, using the change of variables $y:=x+\lambda\xi$, we get
\begin{equation*}
  \inf\Bigl\{ \int_{\R^d} \rho(\xi)\int_{Q_1} f\Bigl(x,x+\lambda\xi,\frac{u(x+\lambda\xi)-u(x)}{\lambda}\Bigr)\, dx\, d\xi : u\in L^p_{\#,M}(Q_1;\R^m) \Bigr\}=f_\lambda(M),
\end{equation*}
that is the thesis.

\medskip

{\em Supercritical case, $\lambda=+\infty$.} In this case the proof is immediate; indeed, by Proposition \ref{prop : main} and the Monotone Convergence Theorem,
\begin{align*}
    \lim_{h\to+\infty}\f_h(M) & = \lim_{h\to+\infty}\int_{B_{T_h}} \int_{Q_1}\int_{Q_1} \rho(\xi)f(x,y,M\xi)\, dx\,dy\,d\xi \\
    & =\int_{\R^d} \int_{Q_1}\int_{Q_1} \rho(\xi) f(x,y,M\xi)\, dx\,dy\,d\xi \\
    & = f_{+\infty}(M)
\end{align*}
for every $M\in \R^{m\times d}$, concluding the proof.
\end{proof}

As a corollary, we extend $(iii)$ of Proposition \ref{prop : cell formula} to kernels that are not necessarily supported on a ball, and also to the case $\lambda=+\infty$. Consequently, we obtain that the densities $\{f_\lambda\}_\lambda$ vary continuously with respect to $\lambda$. 

\begin{corollary}
    Let $M\in \R^{m\times d}$ and let $\{\lambda_j\}_j$ be a positive sequence converging to $\lambda\in[0,+\infty]$. Then 
    \begin{equation*}
    f_{\lambda}(M) = \lim_{j\to+\infty} \inf\Bigl\{ \int_{\R^d} \rho(\xi)\int_{Q_1} f\Bigl(x,x+\lambda_j\xi,\frac{u(x+\lambda_j\xi)-u(x)}{\lambda_j}\Bigr)\, dx\,d\xi : u\in L^p_{\#,M}(Q_1;\R^m) \Bigr\}.
\end{equation*}
    Moreover, the function $\lambda\mapsto f_\lambda(M)$ is continuous in $[0,+\infty]$.
\end{corollary}
\begin{proof}

According to $(ii)$ of Theorem \ref{thm : main} and a change of variables, we need to prove that
\begin{equation*}
    f_\lambda(M)=\lim_{j\to+\infty}f_{\lambda_j}(M).
\end{equation*}

    Let us first suppose $\lambda \in [0,+\infty)$ and let $\{u_j\}_j\subset L^p_{\#,M}(Q_1;\R^m)$ be such that
    \begin{equation*}
       \liminf_{j\to+\infty}f_{\lambda_j}(M)
        = \liminf_{j\to+\infty}  \int_{\R^d} \rho(\xi)\int_{Q_1} f\Bigl(x,x+\lambda_j\xi,\frac{u_j(x+\lambda_j\xi)-u_j(x)}{\lambda_j}\Bigr)\, dx\,d\xi. 
    \end{equation*}
Upon assuming that the mean value of $u_j$ on $Q_1$ is zero for all $j\in\N$, we obtain compactness properties for the sequence $\{u_j\}_j$ reasoning exactly as in Section $3$ for the proofs of the lower bounds in the subcritical and critical cases, and then, also resorting to Theorem \ref{thm : LSC}, we get that
    \begin{align*}
        \liminf_{j\to+\infty}f_{\lambda_j}(M) &\geq \inf\Bigl\{ \int_{\R^d} \rho(\xi)\int_{Q_1} f(x,x,(\nabla u ) \xi)\, dx\,d\xi : u\in W^{1,p}_{\#,M}(Q_1; \R^m) \Bigr\} \\
        & = f_{0}(M)
    \end{align*}
    if $\lambda=0$, and that
    \begin{align*}
  \liminf_{j\to+\infty}f_{\lambda_j}(M) & \geq \inf\Bigl\{ \int_{\R^d} \rho(\xi)\int_{Q_1} f\Bigl(x,x+\lambda\xi,\frac{u(x+\lambda\xi)-u(x)}{\lambda}\Bigr)\, dx\, d\xi : u\in L^p_{\#,M}(Q_1;\R^m) \Bigr\} \\
  & = f_\lambda(M)
\end{align*}
if $\lambda\in(0,+\infty)$. Conversely, to prove that
\begin{equation*}
    \limsup_{j\to+\infty}f_{\lambda_j}(M) \leq f_\lambda(M) ,
\end{equation*}
we observe that
\begin{equation}\label{final limit}
    \limsup_{j\to+\infty}f_{\lambda_j}(M) \leq \limsup_{j\to+\infty}\int_{\R^d} \rho(\xi)\int_{Q_1} f\Bigl(x,x+\lambda_j\xi,\frac{w(x+\lambda_j\xi)-w(x)}{\lambda_j}\Bigr)\, dx\, d\xi
\end{equation}
for any $w\in L^p_{\#,M}(Q_1; \R^m)$, and that
\begin{equation*}
    \int_{\R^d} \rho(\xi)\int_{Q_1} f\Bigl(x,x+\lambda_j\xi,\frac{w(x+\lambda_j\xi)-w(x)}{\lambda_j}\Bigr)\, dx\, d\xi
\end{equation*}
tends to 
\begin{equation*}
    \int_{\R^d} \rho(\xi)\int_{Q_1} f(x,x,(\nabla w ) \xi)\, dx\,d\xi
\end{equation*}
if $\lambda=0$ and $w\in W^{1,p}_{\#,M}(Q_1; \R^m)\cap C^\infty(\R^d;\R^m)$, and tends to
\begin{equation*}
    \int_{\R^d} \rho(\xi)\int_{Q_1} f\Bigl(x,x+\lambda\xi,\frac{w(x+\lambda\xi)-w(x)}{\lambda}\Bigr)\, dx\, d\xi
\end{equation*}
if $\lambda\in(0,+\infty)$ and $w\in L^p_{\#,M}(Q_1; \R^m)\cap C^\infty(\R^d;\R^m)$. Then, letting $w$ be a (almost) minimizer for the minimum problem corresponding to $f_0(M)$ or to $f_\lambda(M)$, we infer the desired inequality by \eqref{final limit}.

To conclude, we consider the case $\lambda=+\infty$. On the one hand, letting $T>0$ sufficiently large so that Remark \ref{rmk super} is valid, we have  
\begin{align*}
    \liminf_{j\to+\infty} f_{\lambda_j}(M) & \geq \lim_{j\to+\infty} \inf\Bigl\{ \int_{B_T} \rho(\xi)\int_{Q_1} f\Bigl(x,x+\lambda_j\xi,\frac{u(x+\lambda_j\xi)-u(x)}{\lambda_j}\Bigr)\, dx\,d\xi : u\in L^p_{\#,M}(Q_1;\R^m) \Bigr\} \\
    & =  \int_{B_T} \int_{Q_1}\int_{Q_1} \rho(\xi)f(x,y,M\xi)\, dx\,dy\,d\xi,
\end{align*}
and then, letting $T\to+\infty$,
\begin{equation*}
    \liminf_{j\to+\infty} f_{\lambda_j}(M) \geq f_{+\infty}(M).
\end{equation*}
Conversely, we recall that, in order to prove the upper bound for the supercritical case in section $3$, we proved that
\begin{equation*}
\rho(\xi)\int_{Q_1} f(x,x+\lambda_j\xi,M\xi)\, dx
\rightharpoonup
    \rho(\xi)\int_{Q_1} \int_{Q_1} f(x,y,M\xi)\, dx\, dy
\end{equation*}
weakly in $L^1(B_T)$ for any $T>0$ large enough. By \eqref{f : growth} and \eqref{rho : moment}, we have that for a.e. $\xi\in\R^d$ and for every $j\in\N$ it holds
\begin{equation*}
    \rho(\xi)\int_{Q_1} f(x,x+\lambda_j\xi,M\xi)\, dx \leq \beta|M|^p\rho(\xi)|\xi|^p \in L^1(\R^d);
\end{equation*}
and therefore, we deduce that
\begin{equation*}
\rho(\xi)\int_{Q_1} f(x,x+\lambda_j\xi,M\xi)\, dx
\rightharpoonup
    \rho(\xi)\int_{Q_1} \int_{Q_1} f(x,y,M\xi)\, dx\, dy
\end{equation*}
weakly in $L^1(\R^d)$. We obtain 
\begin{align*}
f_{+\infty}(M) & = \int_{\R^d} \int_{Q_1}\int_{Q_1} \rho(\xi)f(x,y,M\xi)\, dx\,dy\,d\xi \\
& = \lim_{j\to+\infty} \int_{\R^d} \rho(\xi)\int_{Q_1} f(x,x+\lambda_j\xi,M\xi)\, dx\,d\xi \\
& \geq \limsup_{j\to+\infty} \inf\Bigl\{ \int_{\R^d} \rho(\xi)\int_{Q_1} f\Bigl(x,x+\lambda_j\xi,\frac{u(x+\lambda_j\xi)-u(x)}{\lambda_j}\Bigr)\, dx\,d\xi : u\in L^p_{\#,M}(Q_1;\R^m) \Bigr\} \\
& = \limsup_{j\to+\infty} f_{\lambda_j}(M),
\end{align*}
where we have used $u(x)=Mx$ as a test function for the minimum problems. This  concludes the proof.
\end{proof}

\begin{remark} \label{rmk}
    A milder growth condition from above on $f$ can be required upon enhancing the integrability of $\rho$ at the origin. In particular, we may require in addition that
    \begin{equation}\label{rmk 0}
        \int_{\R^d} \rho(\xi)\, d\xi<
+\infty,   \end{equation}
in order to replace \eqref{f : growth} with
\begin{equation}\label{newGC}
    \alpha|z|^p\leq f(x,y,z)\leq \beta(1+|z|^p) \text{ for almost every } x,y\in \R^d \text{ and for every } z\in \R^m.
\end{equation}
Indeed, there are only a few points in our proofs where a growth condition from above on $f$ is employed. We mention some of them and briefly illustrate how to modify the proofs according to the new set of assumptions.

\smallskip

In the first part of the proof of Proposition \ref{prop : cell formula} we observed that
\begin{equation}\label{rmk 1}
   \Bigl(\frac{\dej}{r}\Bigr)^d\int_{B_T} \rho(\xi)\int_{(Q_{\lceil r/\dej \rceil })_{\lambda_j}(\xi) \setminus (Q_{r/\dej})_{\lambda_j}(\xi)} f(x, x+\lambda_j\xi, M\xi)\, dx\,d\xi
\end{equation}
tends to $0$. Recalling \eqref{set} and using \eqref{newGC}, it is easily seen that \eqref{rmk 1} is estimated from above by
\begin{equation*}
    \Bigl(\frac{\dej}{r}\Bigr)^d|E_j|\beta\int_{B_T} \rho(\xi)\bigl( 1+|M|^p|\xi|^p\bigr)\,d\xi;
\end{equation*}
and since we already proved that
\begin{equation*}
   \Bigl(\frac{\dej}{r}\Bigr)^d|E_{j}| \beta|M|^p\int_{B_T} \rho(\xi)|\xi|^p\,d\xi
\end{equation*}
tends to $0$ as $j\to+\infty$, it suffices to observe that
\begin{equation*}
   \Bigl(\frac{\dej}{r}\Bigr)^d |E_{j}|\beta\int_{B_T} \rho(\xi)\,d\xi
\end{equation*}
vanishes by \eqref{rmk 0} and the fact that $(\dej/r)^d|E_{j}|\to0$.

In the upper bounds for the case $\lambda\in[0,+\infty)$ in Section $3$, \eqref{f : growth} is employed to obtain some uniform bounds from above useful to apply the Dominated Convergence Theorem. As an example, we observe that inequality \eqref{forDCT} is now replaced by
\begin{equation*}
     \rho(\xi) f(x,x,(\nabla u(x))\xi) \leq \beta \rho(\xi)(1+|\xi|^p|\nabla u(x)|^p)\quad \text{ for a.e. } (x,\xi)\in Q_1\times B_T,
\end{equation*}
and the desired estimates are now obtained also resorting to \eqref{rmk 0}. 

In the supercritical case $\lambda=+\infty$, similar simple adaptations are needed in the proofs of Lemma \ref{lemma : supercritical 2} and Lemma \ref{lemma : supercritical 3}. As for the proof of the upper bound, it suffices to observe that
\begin{equation*}
    \rho(\xi)\int_{Q_1} f(x,x+\lambda_j\xi,M\xi)\, dx, \quad j\in\N,
\end{equation*}
is still an equi-integrable sequence since, by \eqref{newGC}, we have that, for every $j\in\N$ and $A\subset B_T$ measurable, it holds
\begin{equation*}
    \int_{A} \rho(\xi)\int_{Q_1} f(x,x+\lambda_j\xi,M\xi)\, dx\,d\xi \leq \beta \Bigl\{\int_A \rho(\xi)\,d\xi+ \int_A \rho(\xi)|M\xi|^p\,d\xi\Bigr\}
\end{equation*}
so that the conclusion follows by \eqref{rmk 0}. Clearly, in this case, the main result holds true assuming \eqref{H1}.
\end{remark}

\medskip

{\textbf{Acknowledgements.} The author wishes to thank Andrea Braides, Irene Fonseca, and Giovanni Leoni for the valuable comments, and acknowledges the hospitality of the Center for Nonlinear Analysis and the Department of Mathematical Sciences at Carnegie Mellon University, Pittsburgh, where the work has been completed. The author is member of Gruppo Nazionale per l'Analisi Matematica, la Probabilità e le loro Applicazioni (GNAMPA) of Istituto Nazionale di Alta Matematica (INdAM).

\bibliographystyle{plain} 
\bibliography{refs} 

@book {DM,
    AUTHOR = {Dal Maso, G.},
     TITLE = {An introduction to {$\Gamma$}-convergence},
    SERIES = {Progress in Nonlinear Differential Equations and their
              Applications},
 PUBLISHER = {Birkh\"auser, Boston},
      YEAR = {1993},
}

@book {BDF,
    AUTHOR = {Braides, A. and Defranceschi, A.},
     TITLE = {Homogenization of multiple integrals},
 PUBLISHER = {Oxford University Press, Oxford},
      YEAR = {1998},
}

@book {FL,
    AUTHOR = {Fonseca, I. and Leoni, G.},
     TITLE = {Modern methods in the Calculus of Variations: $L^p$ spaces},
    SERIES = {Springer Monographs in Mathematics},
    VOLUME = {},
 PUBLISHER = {Springer, New York},
      YEAR = {2007},
}

@book {AABPT,
    AUTHOR = {Alicandro, R. and Ansini, N. and Braides, A. and Piatnitski, A. and Tribuzio, A.},
     TITLE = {A variational theory of convolution-type functionals},
    SERIES = {SpringerBriefs on PDEs and Data Science},
    VOLUME = {},
 PUBLISHER = {Springer, Singapore},
      YEAR = {2023},
}

@article {P1,
    AUTHOR = {Ponce, A. C.},
     TITLE = {A new approach to {S}obolev spaces and connections to
              {$\Gamma$}-convergence},
   JOURNAL = {Calculus of Variations and Partial Differential Equations},
  FJOURNAL = {Calculus of Variations and Partial Differential Equations},
    VOLUME = {19},
      YEAR = {2004},
    NUMBER = {3},
     PAGES = {229-255},
}

@article {P2,
    AUTHOR = {Ponce, A. C.},
     TITLE = {An estimate in the spirit of {P}oincaré's inequality},
   JOURNAL = {Journal of the European Mathematical Society},
  FJOURNAL = {Journal of the European Mathematical Society},
    VOLUME = {6},
      YEAR = {2004},
    NUMBER = {1},
     PAGES = {1-15},
    }

@article {S,
    AUTHOR = {M. Solci},
     TITLE = {Nonlocal-interaction vortices},
   JOURNAL = {SIAM Journal on Mathematical Analysis},
  FJOURNAL = {SIAM Journal on Mathematical Analysis},
    VOLUME = {56},
      YEAR = {2024},
    NUMBER = {3},
     PAGES = {3430-3451},
    }

@article {MQ,
    AUTHOR = {T. Mengesha and Q. Du},
     TITLE = {On the variational limit of a class of nonlocal functionals related to peridynamics},
   JOURNAL = {Nonlinearity},
  FJOURNAL = {Nonlinearity},
    VOLUME = {28},
      YEAR = {2015},
    NUMBER = {11},
     PAGES = {3999-4035},
    }

@article {MT,
    AUTHOR = {C. Mora-Corral and A. Tellini},
     TITLE = {Relaxation of a scalar nonlocal variational problem with a double-well potential},
   JOURNAL = {Calculus of Variations and Partial Differential Equations},
  FJOURNAL = {Calculus of Variations and Partial Differential Equations},
    VOLUME = {59},
      YEAR = {2020},
    NUMBER = {67},
    }

@article {N,
    AUTHOR = {G. Nguetseng},
     TITLE = {A general convergence result for a functional related to the theory of homogenization},
   JOURNAL = {SIAM Journal on Mathematical Analysis},
  FJOURNAL = {SIAM Journal on Mathematical Analysis},
    VOLUME = {20},
      YEAR = {1989},
    NUMBER = {3},
     PAGES = {608--623},
    }

@article {GS,
    AUTHOR = {Gennaioli, L. and Stefani, G.},
     TITLE = {Sharp conditions for the {B}{B}{M} formula and asymptotics of heat content-type energies},
   JOURNAL = {ArXiv: \href{
https://doi.org/10.48550/arXiv.2502.14655
}{2502.14655}},
      YEAR = {2025},
    }

@article {GHS,
    AUTHOR = {Giorgio, R. and Happ, L. and Sch\"onberger, H.},
     TITLE = {Homogenization of nonlocal exchange energies in micromagnetics},
   JOURNAL = {ArXiv: \href{https://doi.org/10.48550/arXiv.2507.13262
}{2507.13262}},
      YEAR = {2025},
    }

@article {KZ,
    AUTHOR = {Kreisbeck, C. and Zappale, E.},
     TITLE = {Loss of double-integral character during relaxation},
   JOURNAL = {SIAM Journal on Mathematical Analysis},
  FJOURNAL = {SIAM Journal on Mathematical Analysis},
    VOLUME = {53},
      YEAR = {2021},
    NUMBER = {1},
     PAGES = {351-385},
    }

@article {DDP,
    AUTHOR = { Davoli, E. and Di Fratta, G. and Pagliari, V.},
     TITLE = { Sharp conditions for the validity of the {B}ourgain–{B}rezis–{M}ironescu formula},
   JOURNAL = {Proceedings of the Royal Society of Edinburgh: Section A Mathematics},
  FJOURNAL = {Proceedings of the Royal Society of Edinburgh: Section A Mathematics},
    VOLUME = {},
      YEAR = {2024},
    NUMBER = {},
     PAGES = {1-24},
    }

@article {BP1,
    AUTHOR = {Braides, A. and Piatnitski, A.},
     TITLE = {Homogenization of random convolution energies},
   JOURNAL = {Journal of the London Mathematical Society},
  FJOURNAL = {Journal of the London Mathematical Society},
    VOLUME = {10},
      YEAR = {2021},
    NUMBER = {2},
     PAGES = {295-319},
    }

@article {BP2,
    AUTHOR = {Braides, A. and Piatnitski, A.},
     TITLE = {Homogenization of quadratic convolution energies in periodically perforated domains},
   JOURNAL = {Advances in Calculus of Variations},
  FJOURNAL = {Advances in Calculus of Variations},
    VOLUME = {15},
      YEAR = {2022},
    NUMBER = {3},
     PAGES = {351--368},
    }

@article {BST,
    AUTHOR = {Braides, A. and Scalabrino, S. and Trifone, C.},
     TITLE = {Homogenization of non-local energies on disconnected sets},
   JOURNAL = {ArXiv: \href{https://doi.org/10.48550/arXiv.2405.09677
}{2405.09677}},
      YEAR = {2024},
    }

@article {AGL,
    AUTHOR = {Alicandro, R. and Gelli, M.S. and Leone, C.},
     TITLE = {Variational analysis of nonlocal {D}irichlet problems in periodically perforated domains},
   JOURNAL = {Calculus of Variations and Partial Differential Equations},
  FJOURNAL = {Calculus of Variations and Partial Differential Equations},
    VOLUME = {64},
      YEAR = {2025},
    NUMBER = {232},
    }

@article {A,
    AUTHOR = {G. Allaire},
     TITLE = {Homogenization and two-scale convergence},
   JOURNAL = {SIAM Journal on Mathematical Analysis},
  FJOURNAL = {SIAM Journal on Mathematical Analysis},
    VOLUME = {23},
      YEAR = {1992},
    NUMBER = {6},
     PAGES = {1482--1518},
    }

@article {BM,
    AUTHOR = {Bellido, J.C. and Mora-Corral, C.},
     TITLE = {Existence for nonlocal variational problems in peridynamics},
   JOURNAL = {SIAM Journal on Mathematical Analysis},
  FJOURNAL = {SIAM Journal on Mathematical Analysis},
    VOLUME = {46},
      YEAR = {2014},
    NUMBER = {1},
     PAGES = {890--916},
    }

@article {BMP,
    AUTHOR = {Bellido, J.C. and Mora-Corral, C. and Pedregal, P.},
     TITLE = {Hyperelasticity as a {$\Gamma$}-limit of peridynamics when the horizon goes to zero},
   JOURNAL = {Calculus of Variations and Partial Differential Equations},
  FJOURNAL = {Calculus of Variations and Partial Differential Equations},
    VOLUME = {54},
      YEAR = {2015},
    }

@article {BP,
    AUTHOR = {Berendsen, J. and Pagliari, V.},
     TITLE = {On the asymptotic behaviour of nonlocal perimeters},
   JOURNAL = {ESAIM: Control, Optimisation and Calculus of Variations},
  FJOURNAL = {ESAIM: Control, Optimisation and Calculus of Variations},
    VOLUME = {25},
      YEAR = {2019},
    }

@article {BBM,
    AUTHOR = {Bourgain, J. and Brezis, H. and Mironescu, P.},
     TITLE = {Another look at {S}obolev spaces},
   JOURNAL = {Optimal
control and partial differential equations. IOS, Amsterdam},
  FJOURNAL = {Optimal
control and partial differential equations. IOS, Amsterdam},
      YEAR = {2001},
     PAGES = { 439–455},
    }

@article {B,
    AUTHOR = {Braides, A.},
     TITLE = {A simplified counterexample to the integral representation of the relaxation of double integrals},
   JOURNAL = {Comptes Rendus - Série Mathématique},
  FJOURNAL = {Comptes Rendus - Série Mathématique},
    VOLUME = {362},
      YEAR = {2024},
    }

@article {BBD,
    AUTHOR = {Braides, A. and Brusca, G. C. and Donati, D.},
     TITLE = {Another look at Elliptic Homogenization},
   JOURNAL = {Milan Journal of Mathematics},
  FJOURNAL = {Milan Journal of Mathematics},
    VOLUME = {92},
      YEAR = {2024},
    }

\end{document}